\documentclass[11pt]{article}

\usepackage{enumerate,xspace}
\usepackage{amsmath,amssymb,wasysym}
\usepackage[all]{xy}
\usepackage{proof}
\usepackage[svgnames]{xcolor}
\usepackage{tikz}

\usepackage{stmaryrd} 
\usepackage{mathtools}
\usepackage{latexsym}
\usepackage{dsfont}
\usepackage{multicol}
\usepackage{cmll}
\usepackage{multirow}
\usepackage{longtable}

\usepackage{lscape}
\usepackage{array}

\usepackage{hyperref} 
\hypersetup{
    colorlinks,
    citecolor=red,
    filecolor=red,
    linkcolor=blue,
    urlcolor=red
}

\newtheorem{observation}{Remark}[section]
\newtheorem{lemma}[observation]{Lemma}  
\newtheorem{theorem}[observation]{Theorem}
\newtheorem{definition}[observation]{Definition}

\newtheorem{remark}[observation]{Remark}

\newtheorem{proposition}[observation]{Proposition} 
\newtheorem{corollary}[observation]{Corollary}

\makeatletter

\newdimen\w@dth

\def\setw@dth#1#2{\setbox\z@\hbox{\scriptsize $#1$}\w@dth=\wd\z@
\setbox\@ne\hbox{\scriptsize $#2$}\ifnum\w@dth<\wd\@ne \w@dth=\wd\@ne \fi
\advance\w@dth by 1.2em}

\def\t@^#1_#2{\allowbreak\def\n@one{#1}\def\n@two{#2}\mathrel
{\setw@dth{#1}{#2}
\mathop{\hbox to \w@dth{\rightarrowfill}}\limits
\ifx\n@one\empty\else ^{\box\z@}\fi
\ifx\n@two\empty\else _{\box\@ne}\fi}}
\def\t@@^#1{\@ifnextchar_ {\t@^{#1}}{\t@^{#1}_{}}}

\def\t@left^#1_#2{\def\n@one{#1}\def\n@two{#2}\mathrel{\setw@dth{#1}{#2}
\mathop{\hbox to \w@dth{\leftarrowfill}}\limits
\ifx\n@one\empty\else ^{\box\z@}\fi
\ifx\n@two\empty\else _{\box\@ne}\fi}}
\def\t@@left^#1{\@ifnextchar_ {\t@left^{#1}}{\t@left^{#1}_{}}}

\def\two@^#1_#2{\def\n@one{#1}\def\n@two{#2}\mathrel{\setw@dth{#1}{#2}
\mathop{\vcenter{\hbox to \w@dth{\rightarrowfill}\kern-1.7ex
                 \hbox to \w@dth{\rightarrowfill}}%
       }\limits
\ifx\n@one\empty\else ^{\box\z@}\fi
\ifx\n@two\empty\else _{\box\@ne}\fi}}
\def\tw@@^#1{\@ifnextchar_ {\two@^{#1}}{\two@^{#1}_{}}}

\def\tofr@^#1_#2{\def\n@one{#1}\def\n@two{#2}\mathrel{\setw@dth{#1}{#2}
\mathop{\vcenter{\hbox to \w@dth{\rightarrowfill}\kern-1.7ex
                 \hbox to \w@dth{\leftarrowfill}}%
       }\limits
\ifx\n@one\empty\else ^{\box\z@}\fi
\ifx\n@two\empty\else _{\box\@ne}\fi}}
\def\t@fr@^#1{\@ifnextchar_ {\tofr@^{#1}}{\tofr@^{#1}_{}}}

\newdimen\W@dth
\def\setW@dth#1#2{\setbox\z@\hbox{$#1$}\W@dth=\wd\z@
\setbox\@ne\hbox{$#2$}\ifnum\W@dth<\wd\@ne \W@dth=\wd\@ne \fi
\advance\W@dth by 1.2em}

\def\T@^#1_#2{\allowbreak\def\N@one{#1}\def\N@two{#2}\mathrel
{\setW@dth{#1}{#2}
\mathop{\hbox to \W@dth{\rightarrowfill}}\limits
\ifx\N@one\empty\else ^{\box\z@}\fi
\ifx\N@two\empty\else _{\box\@ne}\fi}}
\def\T@@^#1{\@ifnextchar_ {\T@^{#1}}{\T@^{#1}_{}}}

\def\T@left^#1_#2{\def\N@one{#1}\def\N@two{#2}\mathrel{\setW@dth{#1}{#2}
\mathop{\hbox to \W@dth{\leftarrowfill}}\limits
\ifx\N@one\empty\else ^{\box\z@}\fi
\ifx\N@two\empty\else _{\box\@ne}\fi}}
\def\T@@left^#1{\@ifnextchar_ {\T@left^{#1}}{\T@left^{#1}_{}}}

\def\Tofr@^#1_#2{\def\N@one{#1}\def\N@two{#2}\mathrel{\setW@dth{#1}{#2}
\mathop{\vcenter{\hbox to \W@dth{\rightarrowfill}\kern-1.7ex
                 \hbox to \W@dth{\leftarrowfill}}%
       }\limits
\ifx\N@one\empty\else ^{\box\z@}\fi
\ifx\N@two\empty\else _{\box\@ne}\fi}}
\def\T@fr@^#1{\@ifnextchar_ {\Tofr@^{#1}}{\Tofr@^{#1}_{}}}

\def\Two@^#1_#2{\def\N@one{#1}\def\N@two{#2}\mathrel{\setW@dth{#1}{#2}
\mathop{\vcenter{\hbox to \W@dth{\rightarrowfill}\kern-1.7ex
                 \hbox to \W@dth{\rightarrowfill}}%
       }\limits
\ifx\N@one\empty\else ^{\box\z@}\fi
\ifx\N@two\empty\else _{\box\@ne}\fi}}
\def\Tw@@^#1{\@ifnextchar_ {\Two@^{#1}}{\Two@^{#1}_{}}}

\def\to{\@ifnextchar^ {\t@@}{\t@@^{}}}
\def\from{\@ifnextchar^ {\t@@left}{\t@@left^{}}}
\def\tofro{\@ifnextchar^ {\t@fr@}{\t@fr@^{}}}
\def\To{\@ifnextchar^ {\T@@}{\T@@^{}}}
\def\From{\@ifnextchar^ {\T@@left}{\T@@left^{}}}
\def\Two{\@ifnextchar^ {\Tw@@}{\Tw@@^{}}}
\def\Tofro{\@ifnextchar^ {\T@fr@}{\T@fr@^{}}}

\makeatother

\title{A Tangent Category Alternative to the Fa\`a di Bruno Construction}
\author{Jean-Simon P. Lemay}

\begin{document}
\allowdisplaybreaks

\maketitle



\begin{abstract} The Fa\`a di Bruno construction, introduced by Cockett and Seely, constructs a comonad $\mathsf{Fa{\grave{a}}}$ whose coalgebras are precisely Cartesian differential categories. In other words, for a Cartesian left additive category $\mathbb{X}$, $\mathsf{Fa{\grave{a}}}(\mathbb{X})$ is the cofree Cartesian differential category over $\mathbb{X}$. Composition in these cofree Cartesian differential categories is based on the Fa\`a di Bruno formula, and corresponds to composition of differential forms. This composition, however, is somewhat complex and difficult to work with. In this paper we provide an alternative construction of cofree Cartesian differential categories inspired by tangent categories. In particular, composition defined here is based on the fact that the chain rule for Cartesian differential categories can be expressed using the tangent functor, which simplifies the formulation of the higher order chain rule. 
\end{abstract}

\section{Introduction}

Cartesian differential categories \cite{blute2009cartesian} were introduced by Blute, Cockett, and Seely to study the coKleisli category of a ``tensor'' differential category \cite{dblute2006differential} and to provide the categorical semantics of Ehrhard and Regnier's differential $\lambda$-calculus \cite{ehrhard2003differential}. In particular, a Cartesian differential category admits a differential combinator $\mathsf{D}$ (see Definition \ref{cartdiffdef} below) whose axioms are based on the basic properties of the directional derivative from multivariable calculus such as the chain rule, linearity, and the symmetry of the second derivative. There are many interesting examples of Cartesian differential categories which originate from a wide range of different fields such as, to list a few, classical differential calculus on Euclidean spaces, functor calculus \cite{bauer2018directional}, and linear logic \cite{dblute2006differential, ehrhard2017introduction}. Generalizations of Cartesian differential categories include a restriction category version \cite{cockett2011differential} to study differentiating partial functions, and Cruttwell's generalized Cartesian differential categories \cite{cruttwell2017Cartesian} which drops the additive structure requirement of a Cartesian differential category. Even more surprising is that there is a notion of a cofree Cartesian differential category! 

Shortly after the introduction of Cartesian differential categories, Cockett and Seely introduced the Fa\`a di Bruno construction \cite{cockett2011faa} which provides a comonad $\mathsf{Fa{\grave{a}}}$ on the category of Cartesian left additive categories (see Definition \ref{CLACdef} below) such that the $\mathsf{Fa{\grave{a}}}$-coalgebras are precisely Cartesian differential categories. Therefore the Eilenberg-Moore category of $\mathsf{Fa{\grave{a}}}$-coalgebras is equivalent to the category of Cartesian differential categories \cite[Theorem 3.2.6]{cockett2011faa}, which implies that there are as many Cartesian differential categories as there are Cartesian left additive categories. 

In particular, for a Cartesian left additive category $\mathbb{X}$, $\mathsf{Fa{\grave{a}}}(\mathbb{X})$ is the cofree Cartesian differential categories over $\mathbb{X}$. This came as a total surprise to Cockett and Seely. Briefly, for a Cartesian left additive category $\mathbb{X}$, $\mathsf{Fa{\grave{a}}}(\mathbb{X})$ has the same objects as $\mathbb{X}$, but its morphisms are sequences $(f_0, f_1, \hdots)$ where $f_n: \underbrace{A \times \hdots \times A}_{n-\text{times}} \to B$ is symmetric and multilinear in its last $n-1$ arguments. The idea here is that $f_n$ should be thought of as a differential form but where we've replaced antisymmetry by symmetry. Therefore, as differentiating twice does not result in zero, $f_{n+1}$ should be thought of as the partial derivative of the non-linear part of $f_n$. Though as far as $\mathsf{Fa{\grave{a}}}(\mathbb{X})$ is concerned, there need not be any relation between $f_n$ and $f_{n+1}$. Composition in $\mathsf{Fa{\grave{a}}}(\mathbb{X})$ \cite[Section 2.1]{cockett2011faa}  is based on the famous Fa\`a di Bruno formula for the higher-order chain rule, and furthermore relates to the idea of composing differential forms. Cruttwell also generalized the Fa\`a di Bruno construction for generalized Cartesian differential categories \cite[Section 2.1]{cruttwell2017Cartesian} to provide a comonad on the category of categories with finite products.  

Surprisingly, while the Fa\`a di Bruno construction is an important result, it seems to have slipped under the radar and not much work has been done with $\mathsf{Fa{\grave{a}}}(\mathbb{X})$. Cruttwell even mentions that: ``A more in-depth investigation of such cofree generalized Cartesian differential categories is clearly required'' \cite[Page 8]{cruttwell2017Cartesian}. Then one is left to wonder why $\mathsf{Fa{\grave{a}}}(\mathbb{X})$ has been left mostly unstudied. One possibility as to why is because the composition of $\mathsf{Fa{\grave{a}}}(\mathbb{X})$ and its differential combinator (which we haven't even mentioned above) is somewhat complex and very combinatorial, making use of symmetric trees \cite[Section 1.1]{cockett2011faa}, and is therefore very notation-heavy. This is due in part that the Fa\`a di Bruno formula itself is very combinatorial in nature and even its simplest expressions depend heavily on combinatorial notation. While we applaud Cockett and Seely for defining and working with the composition of $\mathsf{Fa{\grave{a}}}(\mathbb{X})$, one has to admit that it is indeed difficult to work with. But, as with all universal constructions, the concept of cofree Cartesian differential categories are very important and should not be abandoned! Here we suggest an alternative construction of cofree Cartesian differential categories inspired by tangent categories. 

The concept of a tangent category originate backs to Rosick{\`y} \cite{rosicky1984abstract}, and was later generalized by Cockett and Cruttwell \cite{cockett2014differential}. It should be mentioned that the Fa\`a di Bruno construction predates Cockett and Cruttwell's notion of a tangent category. Of particular importance to us here is that tangent categories come equipped with a tangent bundle endofunctor $\mathsf{T}$, and that every Cartesian differential category is in fact a tangent category. For Cartesian differential categories, the relation between the tangent functor $\mathsf{T}$ and the differential combinator $\mathsf{D}$ is captured by the chain rule, expressed here (\ref{chainrule}), which then provides a very simple expression of the higher-order chain rule, expressed here (\ref{highchainrule}). This higher-order version of the chain rule will be our inspiration for composition in our new presentation of cofree Cartesian differential categories. 

Just like $\mathsf{Fa{\grave{a}}}(\mathbb{X})$, maps of the cofree Cartesian differential category will be special sequences of maps which we call $\mathsf{D}$-sequences (Definition \ref{Ddef}). However, it turns out that most of this construction (such as the differential combinator, composition, etc.) can be done with more generalized sort sequences called pre-$\mathsf{D}$-sequences (Definition \ref{preddef}). Pre-$\mathsf{D}$-sequences can be defined for arbitrary categories with finite products -- no additive structure required. We construct a category of pre-$\mathsf{D}$-sequences (Definition \ref{predcat}), and which provides us with a comonad on the category of categories with finite products (Section \ref{predcomsec}). Briefly, a pre-$\mathsf{D}$-sequence is a sequence of maps $(f_0, f_1, \hdots)$ where $f_n: \underbrace{A \times \hdots \times A}_{2^n-\text{times}} \to B$. The intuition here is that $f_n$ should be thought of as the $n$th total derivative of $f_0$, which is similar to the idea for Fa\`a di Bruno construction but where we also derive the linear arguments. Composition of pre-$\mathsf{D}$-sequences (Definition \ref{predcat} (iv))  is based on the higher-order chain rule using the tangent functor, while the differential of a pre-$\mathsf{D}$-sequence (Definition \ref{tandiffseq} (ii)) is simply the sequence shifted to the left. Compared to the Fa\`a di Bruno construction, this composition and differential can be defined without the need of an additive structure and is quite simple to work with. However, even in the presence of additive structure, the category of pre-$\mathsf{D}$-sequence is not a Cartesian differential category, simply because a pre-$\mathsf{D}$-sequence is too arbitrary a sequence. By considering pre-$\mathsf{D}$-sequences which satisfy extra conditions, based on the axioms of a Cartesian differential category (which now requires an additive structure), we obtain $\mathsf{D}$-sequences, and it follows that the category of $\mathsf{D}$-sequences (Definition \ref{Dcat}) will be the cofree Cartesian differential category. Indeed, $\mathsf{D}$-sequences provide us with a comonad $(\mathcal{D}, \delta, \varepsilon)$ on the category of Cartesian left additive categories (Section \ref{Dcomsec}), such that the $\mathcal{D}$-coalgebras are precisely the Cartesian differential categories (Theorem \ref{mainthm}). Therefore we have that the category of $\mathsf{D}$-sequences of $\mathbb{X}$ will be equivalent as a Cartesian differential category to $\mathsf{Fa{\grave{a}}}(\mathbb{X})$ (Corollary \ref{cor1}). The construction provided here also generalizes to constructing cofree generalized Cartesian differential categories (Appendix \ref{GCDCsec}). Though, as Cartesian differential categories are more prominent then generalized Cartesian differential categories (at the time of writing this paper), we've elected to go straight to building Cartesian differential categories. 

It is always an advantage and very useful to be able to construct and describe a concept in different ways. It allows one to have options to best suit one's needs and interest. Though, it is true that up til now not much work has been done with cofree Cartesian differential categories, and sadly, other then the construction itself, is not done here either. However, we hope that this alternative construction will open the door and inspire future developments in this direction. 

\subparagraph{Conventions:} In this paper, we will use diagrammatic order for composition: this means that the composite map $fg$ is the map which first does $f$ then $g$. 

\section{Cartesian Differential Categories}\label{CDCsec}

In this section, we review Cartesian differential categories \cite{blute2009cartesian}, and a bit of tangent categories \cite{cockett2014differential,cockett2016differential}, to help better understand and motivate pre-$\mathsf{D}$-sequences (Section \ref{PREDsec}) and $\mathsf{D}$-sequences (Section \ref{Dsec}). In particular, we introduce notation and conventions which simplify working with $\mathsf{D}$-sequences. 

\subsection{Cartesian Left Additive Categories}

We begin with the definition of Cartesian left additive categories \cite{blute2009cartesian}. Here ``additive'' is meant being skew enriched over commutative monoids, which in particular means that we do not assume negatives -- this differs from additive categories in the sense of \cite{mac2013categories}. 

\begin{definition} A \textbf{left additive category} \cite[Definition 1.1.1]{blute2009cartesian} is a category such that each hom-set is a commutative monoid, with addition $+$ and zero $0$, such that composition on the {\em left} preserves the additive structure, that is $f(g+h)=fg+fh$ and $f0=0$. A map $h$ in a left additive category is \textbf{additive} \cite{blute2009cartesian} if composition on the right by $h$ preserves the additive structure, that is $(f+g)h=fh+gh$ and $0h=0$. 
\end{definition} 

\begin{definition}\label{CLACdef} A \textbf{Cartesian left additive category} \cite[Definition 1.2.1]{blute2009cartesian} is a left additive category with finite products such that all projections $\pi_i$ are additive. 
\end{definition}

Note that the definition given here of a Cartesian left additive category is slightly different from the one found in \cite[Definition 1.2.1]{blute2009cartesian}. Indeed, we require only that the projections maps be additive and not that pairing of additive is again additive, or equivalently by \cite[Proposition 1.2.2]{blute2009cartesian}, that the diagonal map is additive and that product of additive maps is additive. We now show that requiring the projection maps be additive is sufficient: 

\begin{lemma}\label{claclemma} In a Cartesian left additive category (as defined in Definition \ref{CLACdef}): 
\begin{enumerate}[{\em (i)}]
\item $\langle f ,g \rangle + \langle h, k \rangle = \langle f + h, g+k \rangle$ and $ \langle 0, 0 \rangle = 0$;
\item If $f$ and $g$ are additive then $\langle f,g \rangle$ is additive;
\item The diagonal map $\Delta$ is additive;
\item If $f$ and $g$ are additive then $f \times g$ is additive. 
\end{enumerate}
\end{lemma}
\begin{proof} The proof of $(i)$ is the same as the one found in \cite[Lemma 1.2.3]{blute2009cartesian} and uses only that the projections $\pi_i$ are additive. Then $(ii)$ follows from $(i)$, that is, assuming $f$ and $g$ are additive: 
\[(h+k) \langle f,g \rangle = \langle (h+k) f, (h+k) g \rangle = \langle hf + kf, hg + kg \rangle = \langle hf,hg \rangle + \langle kf,kg \rangle = h\langle f,g \rangle + k\langle f,g \rangle\]
\[0 \langle f,g \rangle = \langle 0f,0g \rangle = \langle 0, 0 \rangle = 0\]
and therefore $\langle f,g \rangle$ is additive. For $(iii)$, \cite[Proposition 1.1.2]{blute2009cartesian} tells us that all identity maps are additive, and therefore by $(ii)$, $\Delta= \langle 1, 1 \rangle$ is additive. For $(iv)$, \cite[Proposition 1.1.2]{blute2009cartesian} also tells us that additive maps are closed under composition, so if $f$ and $g$ are additive, then so is $\pi_0 f$ and $\pi_1 g$. Then again by $(ii)$, $f \times g = \langle \pi_0 f, \pi_1 g \rangle$ is additive. 
\end{proof} 

\subsection{Cartesian Differential Categories}

There are various (but equivalent) ways of expressing the axioms of a Cartesian differential category. We've chosen the one found in \cite[Section 3.4]{cockett2016differential} as it is the most closely relates to tangent categories. In particular, we will express the axioms of a Cartesian differential category using the natural transformations of its tangent category structure \cite{cockett2014differential}. The maps of these natural transformations can be defined without the differential combinator -- though they loose their naturality!  

\begin{definition}\label{cartdiffdef} A \textbf{Cartesian differential category} \cite[Definition 2.1.1]{blute2009cartesian} is a Cartesian left additive category with a combinator $\mathsf{D}$ on maps -- called the \textbf{differential combaintor} -- which written as an inference rule gives: 
\[\infer{\mathsf{D}[f]: A \times A \to B}{f: A \to B}\]
such that $\mathsf{D}$ satisfies the following: 
\begin{enumerate}[{\bf [CD.1]}]
\item $\mathsf{D}[f+g] = \mathsf{D}[f] + \mathsf{D}[g]$ and $\mathsf{D}[0]=0$;
\item $\left (1 \times (\pi_0 + \pi_1) \right) \mathsf{D}[f] = (1 \times \pi_0)\mathsf{D}[f] + (1 \times \pi_1)\mathsf{D}[f]$ and $\langle 1, 0 \rangle \mathsf{D}[f]=0$;
\item $\mathsf{D}[1]=\pi_1$ and $\mathsf{D}[\pi_j] = \pi_1\pi_j$ (with $j \in \lbrace 0,1 \rbrace$);
\item $\mathsf{D}[\langle f, g \rangle] = \langle  \mathsf{D}[f] , \mathsf{D}[g] \rangle$; 
\item $\mathsf{D}[fg] = \langle \pi_0 f, \mathsf{D}[f] \rangle \mathsf{D}[g]$; 
\item $\ell \mathsf{D}^2[f] = \mathsf{D}[f]$ where $\ell := \langle 1,0 \rangle \times \langle 0,1 \rangle: A \times A \to A \times A \times A \times A$;
\item $c \mathsf{D}^2[f] = \mathsf{D}^2[f]$ where $c := 1 \times \langle \pi_1, \pi_0 \rangle \times 1: A \times A \times AÊ\times A \to A \times A \times A \times A$. 
\end{enumerate}
In a Cartesian differential category, a map $f$ is said to be \textbf{linear} \cite[Definition 2.2.1]{blute2009cartesian} if $\mathsf{D}[f]= \pi_1 f$. 
\end{definition}

\begin{remark} \normalfont Note that here we've flipped the convention found in \cite{blute2009cartesian,cockett2016differential,cockett2014differential}. Here we've elected to have the linear argument in the second argument rather then in the first argument. The convention used here follows that of the more recent work on Cartesian differential categories, and is closer to the 
conventions used for the classical notion of the directional derivative such as $\nabla(f)(\vec x) \cdot \vec y$ or $\mathsf{D}[f](\vec x) \cdot \vec y$. \end{remark}

Many examples of Cartesian differential categories can be found throughout the literature. Some intuition for these axioms can be found in \cite[Remark 2.1.3]{blute2009cartesian}. In particular, \textbf{[CD.5]} is the chain rule -- which plays a fundamental role in the main constructions of this paper. We first observe that \textbf{[CD.4]} is in fact redundant (simplifying what we need check later on): 

\begin{lemma}\label{CD4} \textbf{[CD.4]} follows from \textbf{[CD.3]} and \textbf{[CD.5]}. 
\end{lemma}
\begin{proof} Assuming only \textbf{[CD.3]} and \textbf{[CD.5]}, we have that: 
\begin{align*}
\mathsf{D}[\langle f, g \rangle] &= \mathsf{D}[\langle f, g \rangle] \langle \pi_0, \pi_1 \rangle \\
&= \langle \mathsf{D}[\langle f, g \rangle] \pi_0, \mathsf{D}[\langle f, g \rangle]  \pi_1 \rangle \\
&= \langle \langle \pi_0 \langle f, g \rangle, \mathsf{D}[\langle f, g \rangle \rangle \pi_1 \pi_0, \langle \pi_0 \langle f, g \rangle, \mathsf{D}[\langle f, g \rangle \rangle \pi_1 \pi_1 \rangle \\
&= \langle \langle \pi_0 \langle f, g \rangle, \mathsf{D}[\langle f, g \rangle \rangle \mathsf{D}[\pi_0], \langle \pi_0 \langle f, g \rangle, \mathsf{D}[\langle f, g \rangle \rangle  \mathsf{D}[\pi_1] \rangle \\
&= \langle \mathsf{D}[ \langle f, g \rangle \pi_0] , \mathsf{D}[ \langle f, g \rangle \pi_1] \rangle \\
&= \langle \mathsf{D}[f] , \mathsf{D}[ g] \rangle
\end{align*}
\end{proof} 

It is important to note $\langle 1, 0 \rangle$, $\left (1 \times (\pi_0 + \pi_1) \right)$, $\ell$, and $c$ in the axioms  \textbf{[CD.2]}, \textbf{[CD.6]}, and \textbf{[CD.7]}, as they will play fundamental roles in our construction (see Section \ref{Dsec}). First note that they can all be defined without the need of a differential combinator. Second, while $c$ is a natural transformation in the sense that $(f \times f \times f \times f) c = c (f \times f \times f \times f)$, the other three $\langle 1, 0 \rangle$, $\left (1 \times (\pi_0 + \pi_1) \right)$ and $\ell$ are not natural transformations in this same sense (though they are natural for additive maps). However, $\langle 1, 0 \rangle$, $\left (1 \times (\pi_0 + \pi_1) \right)$, $\ell$, and $c$ are natural transformations for the induced \textbf{tangent functor} of a Cartesian differential category. 

While we will not review the full definition of a tangent category (we invite the curious readers to read more on tangent categories here \cite{cockett2016differential,cockett2014differential}), a key observation to this paper is that every differential combinator induces a functor: 

\begin{proposition}\label{tangentfunctorprop} {\cite[Proposition 4.7]{cockett2014differential}} Every Cartesian differential category $\mathbb{X}$ is a tangent category where the \textbf{tangent functor} $\mathsf{T}: \mathbb{X} \to \mathbb{X}$ is defined on objects as $\mathsf{T}(A) := A \times A$ and on morphisms as $\mathsf{T}(f) := \langle \pi_0 f, \mathsf{D}[f] \rangle$. Furthermore, if $f$ is linear then $\mathsf{T}(f) = f \times f$. 
\end{proposition} 

For the tangent functor $\mathsf{T}$, we have that $\langle 1, 0 \rangle$, $\left (1 \times (\pi_0 + \pi_1) \right)$, $\ell$, and $c$ are all natural transformations. In fact, these are all natural transformations of the tangent category structure of a Cartesian differential category \cite[Proposition 4.7]{cockett2014differential}. In tangent category terminology \cite[Definition 2.3]{cockett2014differential}: $\pi_0$ is the projection from the tangent bundle,  $\langle 1, 0 \rangle$ is the zero vector field, $\left (1 \times (\pi_0 + \pi_1) \right)$ is the sum of tangent vectors, $\ell$ is the vertical lift, and $c$ is the canonical flip. We again note that these natural transformations were all defined without the need of a differential combinator, though the differential combinator was necessary for their naturality (in particular in defining the tangent functor). 

Using the tangent functor, the chain rule \textbf{[CD.5]} can be expressed as: 
\begin{equation}\label{chainrule}\begin{gathered} \mathsf{D}[fg] = \mathsf{T}(f) \mathsf{D}[g] \end{gathered}\end{equation}
This then gives a very clean expression for the higher-order chain rule for all $n \in \mathbb{N}$: 
\begin{equation}\label{highchainrule}\begin{gathered} \mathsf{D}^n[fg] = \mathsf{T}^n(f) \mathsf{D}^n[g] \end{gathered}\end{equation}
This simple expression of the higher-order chain rule is key and will allow us to avoid most (if not all) the combinatorial complexities of the Fa\`a di Bruno formula as in \cite{cockett2011faa}. 

\section{Pre-$\mathsf{D}$-Sequences}\label{PREDsec}

In this section we introduce and study pre-$\mathsf{D}$-sequences. Composition of pre-$\mathsf{D}$-sequences -- defined below in (\ref{predcomp}) -- is based on the higher-order chain rule of Cartesian differential categories involving the tangent functor (\ref{highchainrule}). While the category of pre-$\mathsf{D}$-sequences is not a Cartesian differential category, most of the construction of the cofree Cartesian differential category comonad can be done in this weaker setting. In particular, in Section \ref{predcomsec} we provide a comonad on the category of categories with finite products (Proposition \ref{predcom1}), and later extend it to the category of Cartesian left additive categories (Proposition \ref{predcom2}). 

\subsection{Pre-$\mathsf{D}$-Sequences}\label{PREDsubsec}

For a category $\mathbb{X}$ with finite products, consider the endofunctor $\mathsf{P}: \mathbb{X} \to \mathbb{X}$ (where $\mathsf{P}$ is for product) which is defined on objects as $\mathsf{P}(A):= A \times A$ and on maps as $\mathsf{P}(f) := f \times f$. The projection maps give natural transformations $\pi_j: \mathsf{P} \Rightarrow 1_{\mathbb{X}}$ (with $j \in \lbrace 0,1 \rbrace$). 

\begin{definition}\label{preddef} In a category with finite products, a \textbf{pre-$\mathsf{D}$-sequence} from $A$ to $B$, which we denote as $f_\bullet: A \to B$, is a sequence of maps $f_\bullet = (f_0, f_1, f_2, \hdots)$ where the $n$-th term is a map of type $f_n: \mathsf{P}^n(A) \to B$.
\end{definition}

The intuition for pre-$\mathsf{D}$-sequences are sequences of the form $(f, \mathsf{D}[f], \mathsf{D}\!\left[\mathsf{D}[f] \right], \hdots)$. This is similar to the intuition for the maps of the Fa\`a di Bruno construction \cite{cockett2011faa}, but instead of only taking partial derivatives \`a la de Rham cohomolgy, we've taken the full derivative. In the following, we will often wish to prove that two pre-$\mathsf{D}$-sequences are equal to one another, where $f_\bullet = g_\bullet$ means that $f_n = g_n$ for all $n \in \mathbb{N}$. We achieve this by either the ``internal'' method, where we directly show that $f_n = g_n$ for all $n$, or by using the ``external'' method, where we use identities of pre-$\mathsf{D}$-sequences which we will come across throughout this paper. 

One can ``scalar multiply'' pre-$\mathsf{D}$-sequences on the left and on the right by maps of the base category. Given maps $h: A^\prime \to A$ and $k: B \to C$ in $\mathbb{X}$, and a pre-$\mathsf{D}$-sequence $f_\bullet: A \to B$ of $\mathbb{X}$, we define new pre-$\mathsf{D}$-sequences $h \cdot f_\bullet: A^\prime \to B$ and $f_\bullet \cdot k: A \to C$, respectively as follows: 
\[\left( h \cdot f_\bullet \right)_n := \xymatrixcolsep{2pc}\xymatrix{\mathsf{P}^n(A^\prime) \ar[r]^-{\mathsf{P}^n(h)} & \mathsf{P}^n(A) \ar[r]^-{f_n} & B
  } \quad \quad \left(f_\bullet \cdot k \right)_n := \xymatrixcolsep{2pc}\xymatrix{\mathsf{P}^n(A) \ar[r]^-{f_n} & B \ar[r]^-{k} & C
  } \] 
One can easily check the following identities:   
  
\begin{lemma}\label{scalarlem} The following equalities hold: 
\begin{enumerate}[{\em (i)}]
\item $h_1 \cdot (h_2 \cdot f_\bullet) = (h_1 h_2) \cdot f_\bullet$; 
\item $1 \cdot f_\bullet = f_\bullet = f_\bullet \cdot 1$; 
\item $(f_\bullet \cdot k_1) \cdot k_2 = f_\bullet \cdot (k_1 k_2)$; 
\item $h \cdot (f_\bullet \cdot k) = (h \cdot f_\bullet) \cdot k$
\end{enumerate}
\end{lemma}

Even at this early stage, we can already define a differential and tangent: 

\begin{definition}\label{tandiffseq} For a pre-$\mathsf{D}$-sequence $f_\bullet: A \to B$ we define the following two pre-$\mathsf{D}$-sequences: 
\begin{enumerate}[{\em (i)}]
\item Its \textbf{tangent} pre-$\mathsf{D}$-sequence $\mathsf{T}(f_\bullet): \mathsf{P}(A) \to \mathsf{P}(B)$ where: 
\[\mathsf{T}(f_\bullet)_n := \xymatrixcolsep{5pc}\xymatrix{\mathsf{P}^{n+1}(A) \ar[rr]^-{\langle \mathsf{P}^n(\pi_0) f_n , f_{n+1} \rangle} && \mathsf{P}(B)
  } \]
 \item  Its \textbf{differential} pre-$\mathsf{D}$-sequence $\mathsf{D}[f_\bullet]: \mathsf{P}(A) \to B$ where $\mathsf{D}[f_\bullet]_n := f_{n+1}$.
\end{enumerate}
\end{definition}
Looking forward, $\mathsf{D}$ will indeed provide the desired differential combinator, and $\mathsf{T}$ will be its the induced tangent functor (Corollary \ref{DseqCDC}). 

\begin{lemma}\label{Tprop1} The following equalities hold: 
\begin{enumerate}[{\em (i)}]
\item $\mathsf{T}(h \cdot f_\bullet) = \mathsf{P}(h) \cdot \mathsf{T}(f_\bullet)$; 
\item $\mathsf{T}(f_\bullet \cdot k) = \mathsf{T}(f_\bullet) \cdot \mathsf{P}(k)$;
\item $\pi_0 \cdot f_\bullet = \mathsf{T}(f_\bullet) \cdot \pi_0$;
\item $\mathsf{T}(f_\bullet) \cdot \pi_1 = \mathsf{D}[f_\bullet]$;
\item $\mathsf{D}[h \cdot f_\bullet] = \mathsf{P}(h) \cdot \mathsf{D}[f_\bullet]$;
\item $\mathsf{D}[f_\bullet \cdot k] = \mathsf{D}[f_\bullet] \cdot k$. 
\end{enumerate}
\end{lemma}  
\begin{proof} 
\begin{enumerate}[{\em (i)}]
\item Here we use the naturality of $\pi_0$ with respect to $\mathsf{P}$: 
\begin{align*}
\mathsf{T}(h \cdot f_\bullet)_n &=~ \langle \mathsf{P}^n(\pi_0) (h \cdot f_\bullet)_n , (h \cdot f_\bullet)_{n+1} \rangle \\
&=~ \langle \mathsf{P}^n(\pi_0) \mathsf{P}^n(h) f_n, \mathsf{P}^{n+1}(h) f_{n+1} \rangle \\
&=~ \langle \mathsf{P}^{n+1}(h) \mathsf{P}^n(\pi_0) f_n, \mathsf{P}^{n+1}(h) f_{n+1} \rangle \\
&=~ \mathsf{P}^{n+1}(h) \langle \mathsf{P}^n(\pi_0) f_n , f_{n+1} \rangle \\
&=~  \mathsf{P}^{n+1}(h) \mathsf{T}(f_\bullet)_n\\
&=~ \left( \mathsf{P}(h) \cdot \mathsf{T}(f_\bullet) \right)_n
\end{align*}
\item Follows mostly by definition: 
\begin{align*}
\mathsf{T}(f_\bullet \cdot k)_n &=~ \langle \mathsf{P}^n(\pi_0) (f_\bullet \cdot k)_n , (f_\bullet \cdot k)_{n+1} \rangle \\
&=~ \langle \mathsf{P}^n(\pi_0) f_n k , f_{n+1} k \rangle \\
&=~ \langle \mathsf{P}^n(\pi_0) f_n , f_{n+1} \rangle (k \times k) \\
&=~  \mathsf{T}(f_\bullet)_n \mathsf{P}(k) \\
&=~ \left( \mathsf{T}(f_\bullet) \cdot \mathsf{P}(k) \right)_n 
\end{align*}
\item Follows by definition: 
\begin{align*}
\left( \mathsf{T}(f_\bullet) \cdot \pi_0 \right)_n = \mathsf{T}(f_\bullet)_n \pi_0=  \langle \mathsf{P}^n(\pi_0) f_n , f_{n+1} \rangle \pi_0 =  \mathsf{P}^n(\pi_0) f_n  = \left( \pi_0 \cdot f_\bullet  \right)_n 
\end{align*}
\item Follows by definition:
\begin{align*}
\left( \mathsf{T}(f_\bullet) \cdot \pi_1 \right)_n &=~ \mathsf{T}(f_\bullet)_n \pi_1 =\langle \mathsf{P}^n(\pi_0) f_n , f_{n+1} \rangle \pi_1 = f_{n+1} = \mathsf{D}[f_\bullet]_n
\end{align*}
\item Follows from (iv), (i), and Lemma \ref{scalarlem} (iv): 
\begin{align*}
\mathsf{D}[h \cdot f_\bullet] = \mathsf{T}(h \cdot f_\bullet) \cdot \pi_1 = \left(\mathsf{P}(h) \cdot \mathsf{T}(f_\bullet) \right) \cdot \pi_1 = \mathsf{P}(h) \cdot (\mathsf{T}(f_\bullet) \cdot \pi_0) = \mathsf{P}(h) \cdot f_\bullet
\end{align*}
\item Here we use naturality of $\pi_1$, (iv), (ii), and Lemma \ref{scalarlem} (iii): 
\begin{align*}
\mathsf{D}[f_\bullet \cdot k] &=~ \mathsf{T}(f_\bullet \cdot k) \cdot \pi_1 \\
&=~ (\mathsf{T}(f_\bullet) \cdot \mathsf{P}(k) ) \cdot \pi_1 \\
&=~ \mathsf{T}(f_\bullet) \cdot (\mathsf{P}(k) \pi_1) \\
&=~ \mathsf{T}(f_\bullet) \cdot (\pi_1 k) \\
&=~ (\mathsf{T}(f_\bullet) \cdot \pi_1) \cdot k \\
&=~ \mathsf{D}[f_\bullet] \cdot k
\end{align*}
\end{enumerate}
\end{proof} 

\begin{definition}\label{predcat} For a category $\mathbb{X}$ with finite products, we define its category of pre-$\mathsf{D}$-sequences $\overline{\mathcal{D}}[\mathbb{X}]$ as follows:
\begin{enumerate}
\item Objects of $\overline{\mathcal{D}}[\mathbb{X}]$ are objects of $\mathbb{X}$;
\item Maps of $\overline{\mathcal{D}}[\mathbb{X}]$ are pre-$\mathsf{D}$-sequences $f_\bullet: A \to B$;
\item The identity is the pre-$\mathsf{D}$-sequence $i_\bullet: A \to A$ where for $n\geq 1$:
\begin{equation}\label{predid}\begin{gathered}i_n := \xymatrixcolsep{4.5pc}\xymatrix{\mathsf{P}^n(A) \ar[r]^-{\pi_1} & \mathsf{P}^{n-1}(A) \ar[r]^-{\pi_1} & \hdots  \ar[r]^-{\pi_1}& \mathsf{P}(A) \ar[r]^-{\pi_1} & A
  } \end{gathered}\end{equation}
  and $i_0 := 1_A$. 
  \item Composition of pre-$\mathsf{D}$-sequences $f_\bullet: A \to B$ and $g_\bullet: B \to C$ is the pre-$\mathsf{D}$-sequence ${f_\bullet \ast g_\bullet: A \to C}$ where: 
  \begin{equation}\label{predcomp}\begin{gathered}\left(f_\bullet \ast g_\bullet \right)_n := \xymatrixcolsep{5pc}\xymatrix{\mathsf{P}^n(A) \ar[r]^-{\mathsf{T}^n(f_\bullet)_0} &   \mathsf{P}^n(B) \ar[r]^-{g_n} & C 
  } \end{gathered}\end{equation}
\end{enumerate}
\end{definition}

Recall that pre-$\mathsf{D}$-sequences should be thought of as $f_\bullet = (f, \mathsf{D}[f], \hdots)$. By \textbf{[CD.3]}, the identity $i_\bullet$ is precisely $(1, \mathsf{D}[1], \hdots)$. The composition $(f_\bullet \ast g_\bullet)_n$ is the analogue of the higher-order chain rule $\mathsf{D}^n[fg] = \mathsf{T}^n(f) \mathsf{D}^n[g]$. At first glance, $\mathsf{T}^n(f_\bullet)_0$ in the composition may seem intimidating, however by the functorial properties of $\mathsf{T}$, the composition of pre-$\mathsf{D}$-sequences is easy to work with, and will allow us to avoid the combinatorics of \cite{cockett2011faa}.   

Strangely, before proving that $\overline{\mathcal{D}}[\mathbb{X}]$ is a well-defined category, we show the functorial properties of $\mathsf{T}$: 

\begin{lemma}\label{Tfunctor} The following equalities hold: 
\begin{enumerate}[{\em (i)}]
\item $\mathsf{T}(i_\bullet) = i_\bullet$;
\item $\mathsf{T}(f_\bullet \ast g_\bullet)=  \mathsf{T}(f_\bullet) \ast \mathsf{T}(g_\bullet)$. 
\end{enumerate}
\end{lemma}
\begin{proof}
\begin{enumerate}[{\em (i)}]
\item Notice that for each $n \in \mathbb{N}$, we have a natural transformation $i_n: \mathsf{P}^n \Rightarrow 1_{\mathbb{X}}$ and that $i_{n+1} = i_n \pi_1$. Therefore, we obtain the following: 
\begin{align*}
\mathsf{T}(i_\bullet)_n = \langle \mathsf{P}^n(\pi_0) i_n , i_{n+1} \rangle = \langle i_n \pi_0, i_n \pi_1 \rangle = i_n \langle \pi_0, \pi_1 \rangle = i_n
\end{align*}
\item Here we use identities from Lemma \ref{Tprop1}: 
\begin{align*}
\mathsf{T}(f_\bullet \ast g_\bullet)_n &=~ \langle \mathsf{P}^n(\pi_0) \left(f_\bullet \ast g_\bullet\right)_n , \left(f_\bullet \ast g_\bullet\right)_{n+1} \rangle \\
&=~ \langle \mathsf{P}^n(\pi_0) \mathsf{T}^n(f_\bullet)_0 g_n , \mathsf{T}^{n+1}(f_\bullet)_0 g_{n+1} \rangle \\
&=~ \langle \mathsf{T}^n(\pi_0 \cdot f_\bullet)_0 g_n , \mathsf{T}^{n+1}(f_\bullet)_0 g_{n+1} \rangle \tag{Lemma \ref{Tprop1} (i)}\\
&=~ \langle \mathsf{T}^n\left (\mathsf{T}(f_\bullet) \cdot \pi_0 \right)_0 g_n , \mathsf{T}^{n+1}(f_\bullet)_0 g_{n+1} \rangle \tag{Lemma \ref{Tprop1} (iii)} \\
&=~ \langle \left(\mathsf{T}^{n+1}(f_\bullet) \cdot \mathsf{P}^n(\pi_0) \right)_0 g_n , \mathsf{T}^{n+1}(f_\bullet)_0 g_{n+1} \rangle \tag{Lemma \ref{Tprop1} (i)} \\
&=~ \langle  \mathsf{T}^{n+1}(f_\bullet)_0 \mathsf{P}^n(\pi_0) g_n , \mathsf{T}^{n+1}(f_\bullet)_0 g_{n+1} \rangle \\
&=~ \mathsf{T}^{n+1}(f_\bullet)_0 \langle \mathsf{P}^n(\pi_0) g_n , g_{n+1} \rangle \\
&=~ \mathsf{T}^{n}\left(\mathsf{T}(f_\bullet)\right)_0 \mathsf{T}(g_\bullet)_n \\
&=~ \left( \mathsf{T}(f_\bullet) \ast \mathsf{T}(g_\bullet) \right)_n
\end{align*}
\end{enumerate}
\end{proof} 

\begin{proposition} $\overline{\mathcal{D}}[\mathbb{X}]$ is a category.
\end{proposition}
\begin{proof} 
First we prove associativity, $f_\bullet \ast (g_\bullet \ast h_\bullet)= (f_\bullet \ast g_\bullet) \ast h_\bullet$, using Lemma \ref{Tfunctor} (ii): 
\begin{align*}
(f_\bullet \ast (g_\bullet \ast h_\bullet))_n &=~Ê\mathsf{T}^n(f_\bullet)_0 (g_\bullet \ast h_\bullet)_n = \mathsf{T}^n(f_\bullet)_0 \mathsf{T}^n(g_\bullet)_0 h_n = \mathsf{T}^n(f_\bullet \ast g_\bullet)_0 h_n  = ((f_\bullet \ast g_\bullet) \ast h_\bullet)_n 
\end{align*}
Now we prove that $(i_\bullet \ast f_\bullet)=f_\bullet$ using Lemma \ref{Tfunctor} (i):
\begin{align*}
(i_\bullet \ast f_\bullet)_n = \mathsf{T}^n(i_\bullet)_0 f_n = i_0 f_n  = f_n 
\end{align*}
Lastly we show that $(f_\bullet \ast i_\bullet)=f_\bullet$ by looking at when $n=0$ or when $n \leq 1$. The case $n=0$ is automatic since $i_0 = 1$:
\[(f_\bullet \ast i_\bullet)_0 = f_0 i_0 = f_0\]
For the remaining cases, we have that: 
\begin{align*}
(f_\bullet \ast i_\bullet)_{n+1} &=~ \mathsf{T}^{n+1}(f_\bullet)_0 i_{n+1}\\
&=~ \langle \pi_0 \mathsf{T}^{n}(f)_0, \mathsf{T}^{n}(f)_1 \rangle \underbrace{\pi_1 \hdots \pi_1}_{n+1 \text{ times}}\\
&=~ \mathsf{T}^{n}(f)_1 \underbrace{\pi_1 \hdots \pi_1}_{n \text{ times}}\\
&=~ \hdots \\
&=~ \mathsf{T}(f)_{n} \pi_1 \\
&=~ f_{n+1} 
\end{align*}

\end{proof}

By Lemma \ref{Tfunctor}, we also obtain a proper endofunctor $\mathsf{T}:  \overline{\mathcal{D}}[\mathbb{X}] \to  \overline{\mathcal{D}}[\mathbb{X}]$ defined on objects as $\mathsf{T}(A) := \mathsf{P}(A) = A \times A$ and mapping pre-$\mathsf{D}$-sequences to their tangent pre-$\mathsf{D}$-sequence. 

\begin{corollary}\label{Tendo} $\mathsf{T}: \overline{\mathcal{D}}[\mathbb{X}] \to \overline{\mathcal{D}}[\mathbb{X}]$ is a functor.  
\end{corollary} 

Composition of pre-$\mathsf{D}$-sequences is compatible with scalar multiplication (which we leave as an exercise to the reader): 

\begin{lemma}\label{astprop1} The following equalities hold: 
\begin{enumerate}[{\em (i)}]
\item $h \cdot i_\bullet = i_\bullet \cdot h$;
\item $(h \cdot f_\bullet) \ast g_\bullet = h \cdot (f_\bullet \ast g_\bullet)$; 
\item $f_\bullet \ast (g_\bullet \cdot k) = (f_\bullet \ast g_\bullet) \cdot k$; 
\item $(f_\bullet \cdot k) \ast g_\bullet = f_\bullet \ast (k \cdot g_\bullet)$;
\item $(h \cdot i_\bullet) \ast f_\bullet = h \cdot f_\bullet$; 
\item $f _\bullet \ast (i_\bullet \cdot k) = f_\bullet \cdot k$. 
\end{enumerate}
\end{lemma}

We now give a finite product structure on $\overline{\mathcal{D}}[\mathbb{X}]$:

\begin{proposition}\label{preDprod} $\overline{\mathcal{D}}[\mathbb{X}]$ is a category with finite products where:
\begin{enumerate}
\item The product of objects is the product of objects in $\mathbb{X}$; 
\item The projections are the pre-$\mathsf{D}$-sequences $i_\bullet \cdot \pi_0: A \times B \to A$ and $i_\bullet \cdot \pi_1: A \times B \to B$; 
\item The pairing of pre-$\mathsf{D}$-sequences $f_\bullet: C \to A$ and $g_\bullet: C \to B$ is $\langle f_\bullet, g_\bullet \rangle: C \to A \times B$ where $\langle f_\bullet, g_\bullet \rangle_n := \langle f_n, g_n \rangle$;
\item The terminal object is $\mathsf{1}$, the terminal object of $\mathbb{X}$;
\item The unique map to the terminal object is the pre-$\mathsf{D}$-sequence $i_\bullet \cdot \mathsf{t}: A \to \mathsf{1}$. 
\end{enumerate}
\end{proposition}
\begin{proof} Uniqueness of the maps to the terminal object and the pairing of maps follow directly from the finite product structure of $\mathbb{X}$. Therefore, it remains only to show that $\langle f_\bullet, g_\bullet \rangle \ast (i_\bullet \cdot \pi_0) = f_\bullet$ and $\langle f_\bullet, g_\bullet \rangle \ast (i_\bullet \cdot \pi_1) = g_\bullet$. However both follow immediately from Lemma \ref{astprop1} (vi): 
\begin{align*}
\left(\langle f_\bullet, g_\bullet \rangle \ast (i_\bullet \cdot \pi_0)\right)_n &=~ \left(\langle f_\bullet, g_\bullet \rangle \cdot \pi_0 \right)_n = \langle f_\bullet, g_\bullet \rangle_n \pi_0 = \langle f_n, g_n \rangle \pi_0 = f_n
\end{align*}
and similarly for $\langle f_\bullet, g_\bullet \rangle \ast (i_\bullet \cdot \pi_1) = g_\bullet$. 
\end{proof} 

\begin{lemma}\label{astprop2} The following equalities hold: 
\begin{enumerate}[{\em (i)}]
\item $h \cdot \langle f_\bullet, g_\bullet \rangle = \langle h \cdot f_\bullet, h \cdot g_\bullet \rangle$;
\item $f_\bullet \cdot \langle h,k \rangle = \langle f_\bullet \cdot h, f_\bullet \cdot k \rangle$; 
\item $\langle f_\bullet \cdot k_1, g_\bullet \cdot k_2 \rangle = \langle f_\bullet , g_\bullet \rangle \cdot (k_1 \times k_2)$;
\item $\langle f_\bullet, g_\bullet \rangle \cdot \pi_0 = f_\bullet$ and $\langle f_\bullet, g_\bullet \rangle \cdot \pi_1 = g_\bullet$;
\item $f_\bullet \times g_\bullet = \langle \pi_0 \cdot f_\bullet, \pi_1 \cdot g_\bullet \rangle$;
\item $(f_\bullet \times g_\bullet) \cdot \pi_0 = \pi_0 \cdot f_\bullet$ and $(f_\bullet \times g_\bullet) \cdot \pi_1 = \pi_1 \cdot g_\bullet$; 
\item $(f_\bullet \times g_\bullet) \cdot (h \times k) = (f_\bullet \cdot h) \times (g_\bullet \cdot k)$;
\item $\langle \langle f_\bullet,  f^\prime_\bullet \rangle, \langle g_\bullet, g^\prime_\bullet \rangle \rangle \cdot c = \langle \langle f_\bullet, g_\bullet \rangle, \langle f^\prime_\bullet, g^\prime_\bullet \rangle \rangle$. 
\end{enumerate}
\end{lemma}

Notice that Lemma \ref{astprop2} involves the canonical flip $c$ from the differential combinator axiom \textbf{[CD.7]}. This identity will come into play in Section \ref{Dsec}. 

We can now observe the following relations between $\mathsf{D}$ and $\mathsf{T}$:

\begin{proposition}\label{Diff1} The following equalities hold: 
\begin{enumerate}[{\em (i)}]
\item $\mathsf{T}(f_\bullet) = \langle \pi_0 \cdot f_\bullet, \mathsf{D}[f_\bullet] \rangle$;
\item $\mathsf{D}[i_\bullet] = i_\bullet \cdot \pi_1$;
\item $\mathsf{D}[i_\bullet \cdot \pi_j] = i_\bullet \cdot (\pi_1 \pi_j)$; 
\item $\mathsf{D}[f_\bullet \ast g_\bullet] = \mathsf{T}(f_\bullet) \ast \mathsf{D}[g_\bullet]$;
\item $\mathsf{D}[\langle f_\bullet, g_\bullet \rangle] = \langle \mathsf{D}[f_\bullet], \mathsf{D}[g_\bullet] \rangle$.
\end{enumerate}
\end{proposition}
\begin{proof} 
\begin{enumerate}[{\em (i)}]
\item By Lemma \ref{scalarlem} (iv), Lemma \ref{astprop2} (iii) and Lemma \ref{Tprop1} (iii) and (iv) we have that:
\[\mathsf{T}(f_\bullet) =  \mathsf{T}(f_\bullet) \cdot 1 =  \mathsf{T}(f_\bullet) \cdot \langle \pi_0, \pi_1 \rangle =  \langle \mathsf{T}(f_\bullet) \cdot  \pi_0,  \mathsf{T}(f_\bullet) \cdot \pi_1 \rangle = \langle \pi_0 \cdot f_\bullet, \mathsf{D}[f_\bullet] \rangle  \]
\item By the functoriality of $\mathsf{T}$ and Lemma \ref{Tprop1} (iv) we have that: 
\[\mathsf{D}[i_\bullet] = \mathsf{T}(\i_\bullet) \cdot \pi_1 = i_\bullet \cdot \pi_1\]
\item By (iii), Lemma \ref{Tprop1} (vi), and Lemma \ref{scalarlem} (iii) we have that: 
\[\mathsf{D}[i_\bullet \cdot \pi_j] = \mathsf{D}[i_\bullet] \cdot \pi_j = (i_\bullet \cdot \pi_1) \cdot \pi_j =  i_\bullet \cdot (\pi_1 \pi_j) \]
\item By functoriality of $\mathsf{T}$, Lemma \ref{Tprop1} (iv), and Lemma \ref{astprop1} (iii) we have that:
\[\mathsf{D}[f_\bullet \ast g_\bullet] = \mathsf{T}(f_\bullet \ast g_\bullet) \cdot \pi_1 = (\mathsf{T}(f_\bullet) \ast \mathsf{T}(g_\bullet)) \cdot \pi_1 =\mathsf{T}(f_\bullet) \ast (\mathsf{T}(g_\bullet) \cdot \pi_1) =  \mathsf{T}(f_\bullet) \ast \mathsf{D}[g_\bullet]\]
\end{enumerate}
\end{proof}  

Note that Proposition \ref{Diff1} shows that $\mathsf{D}$ already satisfies some of the differential combinator axioms (Definition \ref{cartdiffdef}). Indeed, (ii) and (iii) are \textbf{[CD.3]}, (v) is \textbf{[CD.4]}, and (iv) is the chain rule \textbf{[CD.5]}. Also, (i) says that $\mathsf{T}$ is indeed the tangent functor (Proposition \ref{tangentfunctorprop}) obtained from $\mathsf{D}$. 

\subsection{Comonad of Pre-$\mathsf{D}$-Sequences}\label{predcomsec}

We now show that pre-$\mathsf{D}$-sequences already give us a comonad. Let $\mathsf{CART}$ (for Cartesian) be the category of all categories with finite products and functors between them which preserves the product structure strictly -- which we shall call here a \textbf{strict Cartesian functor}. Explicitly, for the sake of clarity, for a functor $\mathsf{F}$ to be a strict Cartesian functor we must have for objects $\mathsf{F}(A \times B) = \mathsf{F}(A) \times \mathsf{F}(B)$, $\mathsf{F}$ preserves the terminal object, and that for the projections $\mathsf{F}(\pi_j)=\pi_j$. It then follows that $\mathsf{F}(\langle f, g \rangle) = \langle \mathsf{F}(f), \mathsf{F}(g) \rangle$ and that $\mathsf{F}(f \times g) = \mathsf{F}(f) \times \mathsf{F}(g)$. Therefore, $\mathsf{F}\mathsf{P}= \mathsf{P}\mathsf{F}$, for the product functor $\mathsf{P}$ as defined at the beginning of Section \ref{PREDsubsec}. 

Let $\mathsf{F}: \mathbb{X} \to \mathbb{Y}$ be a strict Cartesian functor. Define the functor $\overline{\mathcal{D}}[\mathsf{F}]: \overline{\mathcal{D}}[\mathbb{X}] \to \overline{\mathcal{D}}[\mathbb{Y}]$ on objects as $\overline{\mathcal{D}}[\mathsf{F}](A)=\mathsf{F}(A)$ and for a pre-$\mathsf{D}$-sequence $f_\bullet$ of $\mathbb{X}$, define the pre-$\mathsf{D}$-sequence $\overline{\mathcal{D}}[\mathsf{F}](f_\bullet)$ of $\mathbb{Y}$ by $\overline{\mathcal{D}}[\mathsf{F}](f_\bullet)_n = \mathsf{F}(f_n)$. 

\begin{lemma}\label{DFunctor} $\overline{\mathcal{D}}[\mathsf{F}]: \overline{\mathcal{D}}[\mathbb{X}] \to \overline{\mathcal{D}}[\mathbb{Y}]$ is a strict Cartesian functor. 
\end{lemma}
\begin{proof} That $\overline{\mathcal{D}}[\mathsf{F}]$ preserves the identity follows from that fact that $\mathsf{F}$ preserves projections: 
\[\overline{\mathcal{D}}[\mathsf{F}](i_\bullet)_n = \mathsf{F}(i_n) = \mathsf{F}(\underbrace{\pi_1 \hdots \pi_1}_{n-\text{times}}) = \underbrace{\mathsf{F}(\pi_1) \hdots \mathsf{F}(\pi_1)}_{n-\text{times}} = \underbrace{\pi_1 \hdots \pi_1}_{n-\text{times}} = i_n \]
To show that $\overline{\mathcal{D}}[\mathsf{F}]$ also preserves composition, notice the following compatibility between $\mathsf{F}$ and $\mathsf{T}$: 
\begin{align*}
\mathsf{F}(\mathsf{T}(f_\bullet)_n) &=~ \mathsf{F}(\langle \mathsf{P}^n(\pi_0) f_n, f_{n+1} \rangle)\\
&=~ \langle \mathsf{F}(\mathsf{P}^n(\pi_0) f_n), \mathsf{F}(f_{n+1}) \rangle \tag{$\mathsf{F}$ preserves product structure strictly} \\
&=~ \langle \mathsf{F}(\mathsf{P}^n(\pi_0)) \mathsf{F}(f_n), \mathsf{F}(f_{n+1}) \rangle \\
&=~ \langle \mathsf{P}^n(\mathsf{F}(\pi_0)) \mathsf{F}(f_n), \mathsf{F}(f_{n+1}) \rangle \tag{$\mathsf{F}$ commutes with $\mathsf{P}$} \\
&=~Ê\langle \mathsf{P}^n(\pi_0) \mathsf{F}(f_n), \mathsf{F}(f_{n+1}) \rangle \tag{$\mathsf{F}$ preserves product structure strictly} \\
&=~ \langle \mathsf{P}^n(\pi_0)\overline{\mathcal{D}}[\mathsf{F}](f_\bullet)_n, \overline{\mathcal{D}}[\mathsf{F}](f_\bullet)_{n+1} \rangle \\
&=~ \mathsf{T}\left(\overline{\mathcal{D}}[\mathsf{F}](f_\bullet)\right)_n 
\end{align*}
Then by this above equality and the fact that $\mathsf{F}$ is functor itself, we have that: 
\begin{align*}
\overline{\mathcal{D}}[\mathsf{F}](f_\bullet \ast g_\bullet)_n &=~ \mathsf{F}\left( (f_\bullet \ast g_\bullet)_n \right) \\
&=~ \mathsf{F}\left( \mathsf{T}^n(f_\bullet)_0 g_n\right) \\
&=~  \mathsf{F}\left( \mathsf{T}^n(f_\bullet)_0 \right)  \mathsf{F}\left( g_n\right) \\
&=~ \mathsf{T}^n\left(\overline{\mathcal{D}}[\mathsf{F}](f_\bullet)\right)_0 \overline{\mathcal{D}}[\mathsf{F}](g_\bullet)_n \\
&=~ \left(\overline{\mathcal{D}}[\mathsf{F}](f_\bullet) \ast \overline{\mathcal{D}}[\mathsf{F}](g_\bullet) \right)_n
\end{align*}
Lastly, $\overline{\mathcal{D}}[\mathsf{F}]$ preserves projections since $\mathsf{F}$ preserves projections: 
\[\overline{\mathcal{D}}[\mathsf{F}](i_\bullet \cdot \pi_j)_n= \mathsf{F}((i_\bullet \cdot \pi_j)_n) = \mathsf{F}(i_n \pi_j) = \mathsf{F}(i_n)\mathsf{F}(\pi_j) = i_n \pi_j = (i_\bullet \cdot \pi_j)_n\]
\end{proof} 

\begin{lemma}\label{Dbarfunctor} $\overline{\mathcal{D}}: \mathsf{CART} \to \mathsf{CART}$ is a functor. 
\end{lemma}
\begin{proof} That $\overline{\mathcal{D}}$ is well defined on objects is given by Proposition \ref{preDprod}, while being well defined on maps is given by Lemma \ref{DFunctor}.  It is straightforward to see that by definition $\overline{\mathcal{D}}$ preserves identity functors and composition of functors.
\end{proof} 

We now define a comonad structure on $\overline{\mathcal{D}}$. Starting with the counit, define the functor $\overline{\varepsilon}: \overline{\mathcal{D}}[\mathbb{X}] \to \mathbb{X}$ defined on objects as $\overline{\varepsilon}(A) := A$ and on pre-$\mathsf{D}$-sequences as $\overline{\varepsilon}(f_\bullet) := f_0$.

\begin{lemma}\label{epsilonlemma1} $\overline{\varepsilon}: \overline{\mathcal{D}}[\mathbb{X}] \to \mathbb{X}$ is a strict Cartesian functor. 
\end{lemma}
\begin{proof} That $\varepsilon$ is a functor follows mostly by definition of $\overline{\mathcal{D}}[\mathbb{X}]$: 
\[\varepsilon(i_\bullet) = i_0 = 1 \quad \quad \quad \varepsilon(f_\bullet \ast g_\bullet)= (f_\bullet \ast g_\bullet)_0 = f_0 g_0 = \varepsilon(f_\bullet)\varepsilon(g_\bullet)\]
While for the projection maps we have (recall that $i_0=1$):
\[\varepsilon(i_\bullet \cdot \pi_j) = (i_\bullet \cdot \pi_j)_0 = i_0 \pi_j = \pi_j\]
\end{proof} 

\begin{lemma}\label{natepsilon} $\overline{\varepsilon}: \overline{\mathcal{D}} \Rightarrow 1_{\mathsf{CART}}$ is a natural transformation. 
\end{lemma}
\begin{proof} $\overline{\varepsilon}$ is well defined by Lemma \ref{epsilonlemma1}. We must show that for a strict Cartesian functor ${\mathsf{F}: \mathbb{X} \to \mathbb{Y}}$, the following diagram commutes: 
  \[  \xymatrixcolsep{5pc}\xymatrix{\overline{\mathcal{D}}[\mathbb{X}] \ar[d]_-{\overline{\varepsilon}} \ar[r]^-{\overline{\mathcal{D}}[\mathsf{F}]} &  \overline{\mathcal{D}}[\mathbb{Y}] \ar[d]^-{\overline{\varepsilon}} \\
   \mathbb{X} \ar[r]_-{\mathsf{F}} & \mathbb{Y}
  } \]
  On objects this is clear, while for a pre-$\mathsf{D}$-sequence $f_\bullet$, we have that: 
  \[\overline{\varepsilon}\left( \overline{\mathcal{D}}[\mathsf{F}](f_\bullet) \right) = \overline{\mathcal{D}}[\mathsf{F}](f_\bullet)_0 = \mathsf{F}(f_0) = \mathsf{F}\left(\overline{\varepsilon}(f_\bullet) \right)\]
\end{proof} 

The comultiplication of the comonad is defined as the functor $\overline{\delta}: \overline{\mathcal{D}}[\mathbb{X}] \to \overline{\mathcal{D}}\left[\overline{\mathcal{D}}[\mathbb{X}]\right]$ defined on objects as $\overline{\delta}(A) :=A$ and for a pre-$\mathsf{D}$-sequence $f_\bullet$ of $\mathbb{X}$, $\overline{\delta}(f_\bullet)$ is the pre-$\mathsf{D}$-sequence of $\overline{\mathcal{D}}[\mathbb{X}]$ defined as $\overline{\delta}(f_\bullet)_n := \mathsf{D}^n[f_\bullet]$ for $n \geq 1$ and $\overline{\delta}(f_\bullet)_0=f_\bullet$. Note the similarity between $\overline{\delta}(f_\bullet)$ and the intuition given for pre-$\mathsf{D}$-sequences after Definition \ref{preddef}. 

\begin{lemma}\label{deltalemma1} $\overline{\delta}: \overline{\mathcal{D}}[\mathbb{X}] \to \overline{\mathcal{D}}\left[\overline{\mathcal{D}}[\mathbb{X}]\right]$ is a strict Cartesian functor. 
\end{lemma}
\begin{proof} To help us distinguish between working in $\overline{\mathcal{D}}[\mathbb{X}]$ and $\overline{\mathcal{D}}\left[\overline{\mathcal{D}}[\mathbb{X}]\right]$, we will use the following notation: 
\begin{enumerate}
\item $i_\bullet$ for the identities of $\overline{\mathcal{D}}[\mathbb{X}]$ and $I_\bullet$ for the identities of $\overline{\mathcal{D}}\left[\overline{\mathcal{D}}[\mathbb{X}]\right]$;
\item $\ast$ for composition in $\overline{\mathcal{D}}[\mathbb{X}]$ and $\star$ for composition in $\overline{\mathcal{D}}\left[\overline{\mathcal{D}}[\mathbb{X}]\right]$;
\item $\cdot$ for scalar multiplication between maps of $\overline{\mathcal{D}}[\mathbb{X}]$ and maps of $\mathbb{X}$, and $\odot$ for scalar multiplication between maps of $\overline{\mathcal{D}}\left[\overline{\mathcal{D}}[\mathbb{X}]\right]$ and maps of $\overline{\mathcal{D}}[\mathbb{X}]$. 
\end{enumerate}
Note that by definition (\ref{predid}) and Lemma \ref{astprop1} (i), (v), and (vi), for $I_\bullet$ we have that:
\[ I_n = \underbrace{(i_\bullet \cdot \pi_1) \ast \hdots \ast  (i_\bullet \cdot \pi_1)}_{n-\text{times}} = i_\bullet \cdot (\underbrace{\pi_1 \hdots \pi_1}_{n-\text{times}}) \]
while for the projections $I_\bullet \odot (i_\bullet \cdot \pi_j)$ of $\overline{\mathcal{D}}\left[\overline{\mathcal{D}}[\mathbb{X}]\right]$ we have that:
\[\left( I_\bullet \odot (i_\bullet \cdot \pi_j) \right)_n = I_n \ast (i_\bullet \cdot \pi_j) = \left( i_\bullet \cdot (\underbrace{\pi_1 \hdots \pi_1}_{n-\text{times}}) \right)  \ast (i_\bullet \cdot \pi_j) =  i_\bullet \cdot (\underbrace{\pi_1 \hdots \pi_1}_{n-\text{times}}\pi_j) \]
Now using multiple iterations of Proposition \ref{Diff1} (ii) and (iii), we can easily check that $\overline{\delta}$ preserves the identities and projections: 
\begin{align*}
\overline{\delta}(i_\bullet)_n = \mathsf{D}^n[i_\bullet] =  i_\bullet \cdot (\underbrace{\pi_1 \hdots \pi_1}_{n-\text{times}}) = I_n
\end{align*}
\begin{align*}
\overline{\delta}(i_\bullet \cdot \pi_j)_n = \mathsf{D}^n[i_\bullet \cdot \pi_j] = i_\bullet \cdot (\underbrace{\pi_1 \hdots \pi_1}_{n-\text{times}}\pi_j) = (
I_\bullet \odot (i_\bullet \ast \pi_j) )_n
\end{align*}
To show that $\overline{\delta}$ preserves composition, first consider the product functor $\mathsf{P}: \overline{\mathcal{D}}[\mathbb{X}] \to \overline{\mathcal{D}}[\mathbb{X}]$ as defined at the beginning of Section \ref{PREDsubsec}. In particular using Lemma \ref{Tprop1} (ii), functoriality of $\mathsf{T}$, and Lemma \ref{astprop2} (vii) we have that:
\[\mathsf{T}(i_\bullet \cdot k) = \mathsf{T}(i_\bullet) \cdot \mathsf{P}(k) = i_\bullet \cdot \mathsf{P}(k) = (i_\bullet \times i_\bullet) \cdot (k \times k) = (i_\bullet \cdot k) \times (i_\bullet \cdot k) = \mathsf{P}(i_\bullet \cdot k) \]
Then it follows that: 
\begin{align*}
\mathsf{T}(\overline{\delta}(f_\bullet))_n &=~ \langle \mathsf{P}^n(i_\bullet \cdot \pi_0) \overline{\delta}(f_\bullet)_n, \overline{\delta}(f_\bullet)_{n+1} \rangle\\
&=~ \langle \mathsf{T}^n(i_\bullet \cdot \pi_0) \mathsf{D}^n[f_\bullet], \mathsf{D}^{n+1}[f_\bullet] \rangle\\
&=~ \langle \mathsf{D}^n[\pi_0 \cdot f_\bullet], \mathsf{D}^{n+1}[f_\bullet] \rangle\\
&=~ \mathsf{D}^n[\langle \pi_0 \cdot f_\bullet, \mathsf{D}[f_\bullet] \rangle] \\
&=~ \mathsf{D}^n[\mathsf{T}(f_\bullet)]\\
&=~ \overline{\delta}(\mathsf{T}(f_\bullet))_n
\end{align*}
Finally using this identity that $\overline{\delta}\mathsf{T}= \mathsf{T}\overline{\delta}$ and the higher order version of Proposition \ref{Diff1} (iv), we obtain that:
\begin{align*}
\left(\overline{\delta}(f_\bullet) \star \overline{\delta}(g_\bullet)\right)_n &=~Ê\mathsf{T}^n(\overline{\delta}(f_\bullet))_0 \ast \overline{\delta}(g_\bullet)_n \\
&=~ \overline{\delta}\left(\mathsf{T}^n(f_\bullet)\right)_0 \ast \overline{\delta}(g_\bullet)_n \\
&=~ \mathsf{T}^n(f_\bullet) \ast \mathsf{D}^n[g_\bullet] \\
&=~ \mathsf{D}^n[f_\bullet \ast g_\bullet] \\
&=~Ê\overline{\delta}(f_\bullet \ast g_\bullet)_n
\end{align*}
\end{proof} 

\begin{lemma}\label{natdelta} $\overline{\delta}: \overline{\mathcal{D}} \Rightarrow \overline{\mathcal{D}}\overline{\mathcal{D}}$ is a natural transformation. 
\end{lemma}
\begin{proof} $\overline{\delta}$ is well defined by Lemma \ref{deltalemma1}. We must show the for a strict Cartesian functor ${\mathsf{F}: \mathbb{X} \to \mathbb{Y}}$, the following diagram commutes: 
  \[  \xymatrixcolsep{5pc}\xymatrix{\overline{\mathcal{D}}[\mathbb{X}] \ar[d]_-{\overline{\delta}} \ar[r]^-{\overline{\mathcal{D}}[\mathsf{F}]} &  \overline{\mathcal{D}}[\mathbb{Y}] \ar[d]^-{\overline{\delta}} \\
   \overline{\mathcal{D}}\left[\overline{\mathcal{D}}[\mathbb{X}]\right]\ar[r]_-{\overline{\mathcal{D}}\left[\overline{\mathcal{D}}[\mathsf{F}]\right]} &\overline{\mathcal{D}}\left[\overline{\mathcal{D}}[\mathbb{Y}]\right]
  } \]
  On objects this is clear, while for a pre-$\mathsf{D}$-sequence $f_\bullet$, note that $\mathsf{D}^n[f_\bullet]_m = f_{n+m}$. Then getting our hands dirty a bit with double indexing, we have that: 
  \begin{align*}
 \left( \overline{\mathcal{D}}\left[\overline{\mathcal{D}}[\mathsf{F}]\right]\left(\overline{\delta}(f_\bullet) \right)_n \right)_m &=~
  \left( \overline{\mathcal{D}}[\mathsf{F}](\overline{\delta}(f_\bullet)_n) \right)_m\\
  &=~ \left( \overline{\mathcal{D}}[\mathsf{F}](\mathsf{D}^n[f]) \right)_m \\
  &=~ \mathsf{F}(\mathsf{D}^n[f]_m) \\
  &=~ \mathsf{F}(f_{n+m}) \\
  &=~ \overline{\mathcal{D}}[\mathsf{F}](f_\bullet)_{n+m} \\
  &=~Ê\left(\mathsf{D}^n[\overline{\mathcal{D}}[\mathsf{F}](f_\bullet)] \right)_m\\
  &=~Ê\left( \overline{\delta}\left( \overline{\mathcal{D}}[\mathsf{F}](f_\bullet) \right)_n \right)_m 
\end{align*}
\end{proof} 

Now we check that we have indeed a comonad: 

\begin{proposition}\label{predcom1} $(\overline{\mathcal{D}}, \overline{\delta}, \overline{\varepsilon})$ is a comonad on $\mathsf{CART}$. 
\end{proposition} 
\begin{proof} This is a matter of checking that the following two diagrams commute: 
  \[  \xymatrixcolsep{5pc}\xymatrix{\overline{\mathcal{D}}[\mathbb{X}] \ar@{=}[dr]^-{}\ar[d]_-{\overline{\delta}} \ar[r]^-{\overline{\delta}} &  \overline{\mathcal{D}}\left[\overline{\mathcal{D}}[\mathbb{X}]\right] \ar[d]^-{\overline{\mathcal{D}}[\overline{\varepsilon}]} & \overline{\mathcal{D}}[\mathbb{X}] \ar[d]_-{\overline{\delta}} \ar[r]^-{\overline{\delta}} &  \overline{\mathcal{D}}\left[\overline{\mathcal{D}}[\mathbb{X}]\right] \ar[d]^-{\overline{\delta}} \\
    \overline{\mathcal{D}}\left[\overline{\mathcal{D}}[\mathbb{X}]\right] \ar[r]_-{\overline{\varepsilon}} & \overline{\mathcal{D}}[\mathbb{X}] &  \overline{\mathcal{D}}\left[\overline{\mathcal{D}}[\mathbb{X}]\right]\ar[r]_-{\overline{\mathcal{D}}[\overline{\delta}]} &\overline{\mathcal{D}}\left[\overline{\mathcal{D}}\left[\overline{\mathcal{D}}[\mathbb{X}]\right] \right]
  } \]
These all follow by definition. Starting with the lower triangle: 
\[ \overline{\varepsilon}(\overline{\delta}(f_\bullet)) = \overline{\delta}(f_\bullet)_0 = f_\bullet \]
then the upper triangle: 
\[\overline{\mathcal{D}}[\overline{\varepsilon}](\overline{\delta}(f_\bullet))_n = \overline{\varepsilon}(\overline{\delta}(f_\bullet)_n)= \overline{\varepsilon}(\mathsf{D}^n[f_\bullet]) =\mathsf{D}^n[f_\bullet]_0 = f_n  \]
and lastly the right square -- getting our hands dirty again with double indexing: 
\begin{align*}
\left(\overline{\mathcal{D}}[\overline{\delta}](\overline{\delta}(f_\bullet))_n \right)_m &=~ \left( \overline{\delta}(\overline{\delta}(f_\bullet)_n) \right)_m \\
&=~\left(Ê\overline{\delta}(\mathsf{D}^n[f_\bullet])\right)_m \\
&=~Ê\mathsf{D}^{n+m}[f_\bullet] \\
&=~ \overline{\delta}(f_\bullet)_{n+m} \\
&=~ \mathsf{D}^n[\overline{\delta}(f_\bullet)]_m \\
&=~ \left(\overline{\delta}(\overline{\delta}(f_\bullet))_n \right)_m
\end{align*}
\end{proof} 

\subsection{Cartesian Left Additive Structure of Pre-$\mathsf{D}$-Sequences}\label{predaddsec}

When the base category is a Cartesian left additive category, one can also sum pre-$\mathsf{D}$-sequences pointwise. 

 \begin{proposition} If $\mathbb{X}$ is a Cartesian left additive category, then so is $\overline{\mathcal{D}}[\mathbb{X}]$ where:
\begin{enumerate}
\item The zero map is the pre-$\mathsf{D}$-sequence $0_\bullet: A \to B$ where $0_n = 0$;
\item The sum of pre-$\mathsf{D}$-sequences $f_\bullet: A \to B$ and $g_\bullet: A \to B$ is $f_\bullet + g_\bullet: A \to B$ where $\left( f_\bullet + g_\bullet \right)_n := f_n + g_n$. 
\end{enumerate}
\end{proposition}
\begin{proof}
This is straightforward by the Cartesian left additive structure of $\mathbb{X}$. 
\end{proof} 

\begin{lemma}\label{astprop3} The following equalities hold: 
\begin{enumerate}[{\em (i)}]
\item $h \cdot 0_\bullet = 0_\bullet$;
\item $f_\bullet \cdot 0 = 0_\bullet$;
\item $f_\bullet \cdot (h+k) = (f_\bullet \cdot h) + (g_\bullet \cdot k)$; 
\item $h \cdot (f_\bullet + g_\bullet) = (h \cdot f_\bullet) + (h \cdot g_\bullet)$; 
\item If $k$ is additive, $0_\bullet \cdot k = 0_\bullet$;
\item If $k$ is additive, $(f_\bullet + g_\bullet) \cdot k = (f_\bullet \cdot k) + (g_\bullet \cdot k)$;
\item $f_\bullet \cdot \langle 1, 0 \rangle = \langle f_\bullet, 0_\bullet \rangle$;
\item $\langle f_\bullet, \langle g_\bullet, g^\prime_\bullet \rangle \rangle \cdot (1 \times (\pi_0 + \pi_1)) = \langle f_\bullet, g_\bullet + g^\prime_\bullet \rangle$;
\item $\langle f_\bullet, g_\bullet  \rangle \cdot \ell = \langle \langle f_\bullet, 0_\bullet \rangle, \langle 0_\bullet, g_\bullet \rangle \rangle$. 
\end{enumerate}
\end{lemma} 

Notice that Lemma \ref{astprop3} (vii), (viii), and (ix) involve $\langle 1, 0 \rangle$, $\left (1 \times (\pi_0 + \pi_1) \right)$ and $\ell$ from the differential combinator axioms \textbf{[CD.2]} and \textbf{[CD.6]}. These, along with Lemma \ref{astprop1} (viii), will be crucial tools for certain proofs in Section \ref{Dsec}. 

The additive structure is also compatible with the differential and tangent of pre-$\mathsf{D}$-sequences: 

\begin{proposition}\label{Diff2} The following equalities hold: 
\begin{enumerate}[{\em (i)}]
\item $\mathsf{D}[0_\bullet] = 0_\bullet$;
\item $\mathsf{D}[f_\bullet + g_\bullet]= \mathsf{D}[f_\bullet] + \mathsf{D}[g_\bullet]$; 
\item $\mathsf{T}(0_\bullet) = 0_\bullet$;
\item $\mathsf{T}(f_\bullet + g_\bullet)= \mathsf{T}(f_\bullet) + \mathsf{T}(g_\bullet)$. 
\end{enumerate}
\end{proposition} 
\begin{proof} 
\begin{enumerate}[{\em (i)}]
\item Follows by definition: 
\[\mathsf{D}[0_\bullet]_n = 0_{n+1} = 0 = 0_n\]
\item Also follows by definition: 
\[\mathsf{D}[f_\bullet + g_\bullet] = \left( f_\bullet + g_\bullet \right)_{n+1} = f_{n+1} + g_{n+1} = \mathsf{D}[f_\bullet]_n + \mathsf{D}[g_\bullet]_n = (\mathsf{D}[f_\bullet] + \mathsf{D}[g_\bullet])_n \]
\item Using (i), Proposition \ref{Diff1} (i), Lemma \ref{astprop3} (i), and Lemma \ref{claclemma} (i) we have that: 
\[\mathsf{T}(0_\bullet)= \langle \pi_0 \cdot 0_\bullet, \mathsf{D}[0_\bullet] \rangle = \langle 0_\bullet, 0_\bullet \rangle = 0_\bullet \]
\item Using (ii), Proposition \ref{Diff1} (i), Lemma \ref{astprop3} (iv), and Lemma \ref{claclemma} (i) we have that: 
\[\mathsf{T}(f_\bullet + g_\bullet)= \langle \pi_0 \cdot (f_\bullet + g_\bullet), \mathsf{D}[f_\bullet + g_\bullet] \rangle  = \langle \pi_0 \cdot f_\bullet, \mathsf{D}[f_\bullet] \rangle + \langle \pi_0 \cdot g_\bullet, \mathsf{D}[g_\bullet] \rangle = \mathsf{T}(f_\bullet) + \mathsf{T}(g_\bullet)\]
\end{enumerate}
\end{proof} 

Note that we have shown another differential combinator axiom: Proposition \ref{Diff2} (i) and (ii) are precisely \textbf{[CD.1]}. Therefore the only axioms remaining are \textbf{[CD.2]}, \textbf{[CD.6]}, and \textbf{[CD.7]}. To obtaining these last three axioms, we will have to consider special kinds of pre-$\mathsf{D}$-sequences which we shall call $\mathsf{D}$-sequences (Definition \ref{Ddef}) and are discussed in the next section.

The comonad from Section \ref{predcomsec} extends to the category of Cartesian left additive categories. Let $\mathsf{CLAC}$ be the category of Cartesian left additive category and strict Cartesian functors between them which preserve the additive structure -- which we will call \textbf{strict Cartesian left additive functors}. Again, to make things explicit, a strict Cartesian functor $\mathsf{F}$ preserves the additive structure if $\mathsf{F}(f+g)=\mathsf{F}(f)+\mathsf{F}(g)$ and $\mathsf{F}(0)=0$. 

\begin{lemma}\label{Daddfunc} If $\mathsf{F}$ is a strict Cartesian left additive functor, then so is $\overline{\mathcal{D}}[\mathsf{F}]$. 
\end{lemma}
\begin{proof} By Lemma \ref{DFunctor}, we need only show that $\overline{\mathcal{D}}[\mathsf{F}]$ preserves the additive structure -- which follows from the fact that $\mathsf{F}$ does: 
\[\overline{\mathcal{D}}[\mathsf{F}](0_\bullet)_n = \mathsf{F}(0_n) = \mathsf{F}(0)=0 \]
\begin{align*}
\overline{\mathcal{D}}[\mathsf{F}](f_\bullet + g_\bullet)_n &= \mathsf{F}((f_\bullet + g_\bullet)_n)\\
&=~ \mathsf{F}(f_n + g_n)= \mathsf{F}(f_n) + \mathsf{F}(g_n)\\
&=~\overline{\mathcal{D}}[\mathsf{F}](f_\bullet)_n + \overline{\mathcal{D}}[\mathsf{F}](g_\bullet)_n\\
&=~ \left( \overline{\mathcal{D}}[\mathsf{F}](f_\bullet) + \overline{\mathcal{D}}[\mathsf{F}](g_\bullet) \right)_n
\end{align*}
\end{proof} 

Abusing notation, we have that $\overline{\mathcal{D}}$ is a well-defined endofunctor on $\mathsf{CLAC}$. 

\begin{lemma}\label{epsilonadd} For Cartesian left additive categories, $\overline{\varepsilon}$ and $\overline{\delta}$ are both strict Cartesian left additive functors. 
\end{lemma}

Therefore we obtain a comonad on $\mathsf{CLAC}$: 

\begin{proposition}\label{predcom2} $(\overline{\mathcal{D}}, \overline{\delta}, \overline{\varepsilon})$ is a comonad on $\mathsf{CLAC}$.
\end{proposition} 

\section{$\mathsf{D}$-Sequences and Cofree Cartesian Differential Categories}\label{Dsec}

\subsection{$\mathsf{D}$-Sequences} In this section we introduce $\mathsf{D}$-sequences and use the category of $\mathsf{D}$-sequences to construct the cofree Cartesian differential categories comonad on the category of Cartesian left additive categories. 

\begin{definition} \label{Ddef} For a Cartesian left additive category, a \textbf{$\mathsf{D}$-sequence} is a pre-$\mathsf{D}$-sequence ${f_\bullet}$ such that for each $n \in \mathbb{N}$ the following equalities hold: 
\begin{enumerate}[{\bf [DS.1]}]
\item $\langle 1,0 \rangle \cdot \mathsf{D}^{n+1}[f_\bullet] = 0_\bullet$; 
\item $\left(1 \times (\pi_0 + \pi_1) \right) \cdot \mathsf{D}^{n+1}[f_\bullet] = \left((1 \times \pi_0) \cdot \mathsf{D}^{n+1}[f_\bullet] \right) + \left((1 \times \pi_1) \cdot \mathsf{D}^{n+1}[f_\bullet] \right)$; 
\item $\ell \cdot \mathsf{D}^{n+2}[f_\bullet] = \mathsf{D}^{n+1}[f_\bullet]$;
\item $c \cdot \mathsf{D}^{n+2}[f_\bullet] = \mathsf{D}^{n+2}[f_\bullet]$;
\end{enumerate}
where recall that $\ell: \mathsf{P}(A) \to \mathsf{P}^2(A)$ and $c: \mathsf{P}^2(A) \to \mathsf{P}^2(A)$ (from Definition \ref{cartdiffdef}) are defined respectively as $\ell := \langle 1,0 \rangle \times \langle 0,1 \rangle$ and $c := 1 \times \langle \pi_1, \pi_0 \rangle \times 1$. 
\end{definition}

Before providing some intuition on $\mathsf{D}$-sequences, we provide an equivalent definition which gives a slightly more explicit description of the maps of the sequence themselves: 

\begin{proposition} For a pre-$\mathsf{D}$-sequence $f_\bullet: A \to B$ of a Cartesian left additive category, the following are equivalent: 
\begin{enumerate}[{\em (i)}]
\item $f_\bullet$ is a $\mathsf{D}$-sequence; 
\item For each $n \in \mathbb{N}$ and $k \leq n$, $f_\bullet$ satisfies the following equalities: 
\begin{enumerate}[{\bf [DS.$1^\prime$]}]
\item $\mathsf{P}^k(\langle 1, 0 \rangle) f_{n+1} = 0$; 
\item $\mathsf{P}^k\left(1 \times (\pi_0 + \pi_1) \right) f_{n+1} = \mathsf{P}^k(1 \times \pi_0) f_{n+1} + \mathsf{P}^k(1 \times \pi_1) f_{n+1}$;
\item $\mathsf{P}^k(\ell) f_{n+2} =   f_{n+1}$ with $\ell: \mathsf{P}^{n-k+1}(A) \to \mathsf{P}^{n-k+1}(A)$;
\item $\mathsf{P}^k(c) f_{n+2} =   f_{n+2}$ with $c: \mathsf{P}^{n-k+2}(A) \to \mathsf{P}^{n-k+2}(A)$. 
\end{enumerate}
\end{enumerate}
\end{proposition}
\begin{proof} In both directions, we use the trick that $\mathsf{D}^n[f]_m = f_{n+m}$. \\Ê\\
$(i) \Rightarrow (ii)$: Suppose that $f_\bullet$ is a $\mathsf{D}$-sequence. For the following, let $k \leq n$: 
 \begin{enumerate}[{\bf [DS.$1^\prime$]}]
\item Here we use {\bf [DS.1]} at $n-k$:
\[\mathsf{P}^k(\langle 1, 0 \rangle) f_{n+1} = \mathsf{P}^k(\langle 1, 0 \rangle) \mathsf{D}^{n-k+1}[f_\bullet]_k = \left( \langle 1, 0 \rangle \cdot \mathsf{D}^{n-k+1}[f_\bullet] \right)_k = (0_\bullet)_k = 0 \]
\item  Here we use {\bf [DS.2]} at $n-k$:
\begin{align*}
\mathsf{P}^k(\left(1 \times (\pi_0 + \pi_1)  \right)) f_{n+1} &=~ \mathsf{P}^k\left(1 \times (\pi_0 + \pi_1)  \right) \mathsf{D}^{n-k+1}[f_\bullet]_k \\
&=~ \left( \left(1 \times (\pi_0 + \pi_1)  \right) \cdot \mathsf{D}^{n-k+1}[f_\bullet] \right)_k \\
&=~ \left(\left((1 \times \pi_0) \cdot \mathsf{D}^{n-k+1}[f_\bullet] \right) + \left((1 \times \pi_1) \cdot \mathsf{D}^{n-k+1}[f_\bullet] \right) \right)_k \\
&=~ \left((1 \times \pi_0) \cdot \mathsf{D}^{n-k+1}[f_\bullet] \right)_k  + \left((1 \times \pi_1) \cdot \mathsf{D}^{n-k+1}[f_\bullet] \right)_k \\
&=~ \mathsf{P}^k(1 \times \pi_0) \mathsf{D}^{n-k+1}[f_\bullet]_k + \mathsf{P}^k(1 \times \pi_1) \mathsf{D}^{n-k+1}[f_\bullet]_k \\
&=~ \mathsf{P}^k(1 \times \pi_0) f_{n+1} + \mathsf{P}^k(1 \times \pi_1) f_{n+1}
\end{align*}
\item Here we use {\bf [DS.3]} at $n-k$:
\[\mathsf{P}^k(\ell) f_{n+2} = \mathsf{P}^k(\ell) \mathsf{D}^{n-k+2}[f_\bullet]_k = \left( \ell \cdot \mathsf{D}^{n-k+2}[f_\bullet] \right)_k = \mathsf{D}^{n-k+1}[f_\bullet]_k =  f_{n+1}\] 
\item Here we use {\bf [DS.4]} at $n-k$:
\[\mathsf{P}^k(c) f_{n+2} = \mathsf{P}^k(c) \mathsf{D}^{n-k+2}[f_\bullet]_k = \left( c \cdot \mathsf{D}^{n-k+2}[f_\bullet] \right)_k = \mathsf{D}^{n-k+2}[f_\bullet]_k =  f_{n+2}\] 
\end{enumerate}
$(ii) \Rightarrow (i)$: Suppose that $f_\bullet$ satisfies {\bf [DS.$1^\prime$]} to {\bf [DS.$4^\prime$]} for each $n \in \mathbb{N}$ and $k \leq n$. 
\begin{enumerate}[{\bf [DS.1]}]
\item Here we use {\bf [DS.$1^\prime$]} with $k \leq n+k$: 
\[(\langle 1, 0 \rangle \cdot \mathsf{D}^{n+1}[f_\bullet])_k = \mathsf{P}^k(\langle 1, 0 \rangle) \mathsf{D}^{n+1}[f_\bullet]_k = \mathsf{P}^k(\langle 1, 0 \rangle) f_{n+k+1} = 0 = 0_k\]
\item Here we use {\bf [DS.$2^\prime$]} with $k \leq n+k$: 
\begin{align*}
(\left(1 \times (\pi_0 + \pi_1)  \right) \cdot \mathsf{D}^{n+1}[f_\bullet])_k &=~ \mathsf{P}^k\left(1 \times (\pi_0 + \pi_1)  \right) \mathsf{D}^{n+1}[f_\bullet]_k\\
&=~ \mathsf{P}^k\left(1 \times (\pi_0 + \pi_1)  \right) f_{n+k+1} \\
&=~Ê\mathsf{P}^k(1 \times \pi_0) f_{n+k+1} + \mathsf{P}^k(1 \times \pi_1) f_{n+k+1} \\
&=~Ê\mathsf{P}^k(1 \times \pi_0) \mathsf{D}^{n+1}[f_\bullet]_k + \mathsf{P}^k(1 \times \pi_1)\mathsf{D}^{n+1}[f_\bullet]_k \\
&=~Ê\left((1 \times \pi_0) \cdot \mathsf{D}^{n+1}[f_\bullet] \right)_k + \left((1 \times \pi_1) \cdot \mathsf{D}^{n+1}[f_\bullet] \right)_k \\
&=~ \left( \left((1 \times \pi_0) \cdot \mathsf{D}^{n+1}[f_\bullet] \right) + \left((1 \times \pi_1) \cdot \mathsf{D}^{n+1}[f_\bullet] \right) \right)_k 
\end{align*}
\item Here we use {\bf [DS.$3^\prime$]} with $k \leq n+k$: 
\[(\ell \cdot \mathsf{D}^{n+2}[f_\bullet])_k = \mathsf{P}^k(\ell) \mathsf{D}^{n+2}[f_\bullet]_k = \mathsf{P}^k(\ell) f_{n+k+2} = f_{n+k+1} = \mathsf{D}^{n+1}[f_\bullet]_k\]
\item Here we use {\bf [DS.$4^\prime$]} with $k \leq n+k$: 
\[(c \cdot \mathsf{D}^{n+2}[f_\bullet])_k = \mathsf{P}^k(c) \mathsf{D}^{n+2}[f_\bullet]_k = \mathsf{P}^k(c) f_{n+k+2} = f_{n+k+2} = \mathsf{D}^{n+2}[f_\bullet]_k\]
\end{enumerate}
\end{proof}

Now to provide some explanation on the axioms of $\mathsf{D}$-sequences. The axioms {\bf [DS.$1$]} to {\bf [DS.$4$]} are analogues of higher-order versions of \textbf{[CD.2]}, \textbf{[CD.6]}, and \textbf{[CD.7]} -- which are respectively:  
\begin{description}
\item[\textbf{[CD.2.a]} ] $(\left(1 \times (\pi_0 + \pi_1)  \right) \mathsf{D}^{n+1}[f] = (1 \times \pi_0)\mathsf{D}^{n+1}[f] + (1 \times \pi_1)\mathsf{D}^{n+1}[f]$ 
\item[\textbf{[CD.2.b]} ] $\langle 1, 0 \rangle \mathsf{D}^{n+1}[f]=0$ 
\item[\textbf{[CD.6]} ] $\ell \mathsf{D}^{n+2}[f] = \mathsf{D}^{n+1}[f]$
\item[\textbf{[CD.7]} ] $c \mathsf{D}^{n+2}[f] = \mathsf{D}^{n+2}[f]$
\end{description}
As for {\bf [DS.$1^\prime$]} to {\bf [DS.$4^\prime$]}, recall that one should think of pre-$\mathsf{D}$-sequences as $(f, \mathsf{D}[f], \mathsf{D}^2[f], \hdots)$. By the higher-order chain rule (\ref{highchainrule}), there are in fact $k \leq n$ possible equalities of the order $n$ versions of \textbf{[CD.2]}, \textbf{[CD.6]}, and \textbf{[CD.7]}. For example, consider \textbf{[CD.6]}: 
\begin{align*}
\mathsf{T}^k(\ell) \mathsf{D}^{n+2}[f] = \mathsf{T}^k(\ell) \mathsf{D}^k\left[\mathsf{D}^2\left[\mathsf{D}^{n-k}[f] \right] \right] = \mathsf{D}^k\left[ \ell \mathsf{D}^2\left[\mathsf{D}^{n-k}[f] \right] \right] = \mathsf{D}^k\left[ \mathsf{D}\left[\mathsf{D}^{n-k}[f] \right] \right] = \mathsf{D}^{n+1}[f]
\end{align*}
and since $\ell$ is linear, $\mathsf{T}(\ell)=\ell \times \ell = \mathsf{P}(\ell)$, and we obtain \textbf{[DS.$3^\prime$]} that $\mathsf{P}^k(\ell) f_{n+2} =   f_{n+1}$. The rest are obtained in similar fashions and are derived in the proof of Lemma \ref{omegalemma}. 

To compare with the Fa\`a di Bruno construction \cite{cockett2011faa}: the requirements of sequences for the Fa\`a di Bruno construction were multi-additivity and symmetry in the last $n$-arguments. 
Here \textbf{[DS.$1^\prime$]} and \textbf{[DS.$2^\prime$]} correspond to the multi-additivity in certain arguments, and \textbf{[DS.$4^\prime$]} is symmetry in those same arguments. The major difference between $\mathsf{D}$-sequences and sequences of the Fa\`a di Bruno construction is \textbf{[DS.$3^\prime$]}. For the Fa\`a di Bruno construction, there was no necessary connection between $f_{n+1}$ and $f_{n+2}$ for arbitrary sequences, while for a $\mathsf{D}$-sequence we ask that there be a relation between the $f_n$. This extra requirement shouldn't be surprising as we are working with the full derivative which involves differentiating the linear argument of $\mathsf{D}[f]$, instead of only partial derivatives. Thus an added requirement explaining this phenomena was to be expected. In summary: in exchange for a simpler composition, we require an added axiom. 

There is a bit of work to do in order to show that the category of $\mathsf{D}$-sequences is well defined: in particular proving that the composite of $\mathsf{D}$-sequences is again a $\mathsf{D}$-sequence. We first observe the following (which we leave to the reader to check for themselves as they are all straightforward): 

\begin{lemma}\label{Dlemma1} For a Cartesian left additive category: 
\begin{enumerate}[{\em (i)}]
\item $0_\bullet$ is a $\mathsf{D}$-sequence; 
\item $i_\bullet$ is a $\mathsf{D}$-sequence; 
\item If $f_\bullet$ is a $\mathsf{D}$-sequence and $h$ is additive then $h \cdot f_\bullet$ is a $\mathsf{D}$-sequence;
\item If $f_\bullet$ is a $\mathsf{D}$-sequence and $k$ is additive then $f_\bullet \cdot k$ is a $\mathsf{D}$-sequence;
\item $i_\bullet \cdot \pi_i$ is a $\mathsf{D}$-sequence; 
\item If $f_\bullet$ and $g_\bullet$ are $\mathsf{D}$-sequences then $\langle f_\bullet, g_\bullet \rangle$ is a $\mathsf{D}$-sequence; 
\item If $f_\bullet$ and $g_\bullet$ are $\mathsf{D}$-sequences then $f_\bullet + g_\bullet$ is a $\mathsf{D}$-sequence. 
\end{enumerate}
\end{lemma}

Next we show that $\mathsf{D}$-sequences are closed under the differential and tangent.

\begin{proposition}\label{Dprop1}  If $f_\bullet$ is a $\mathsf{D}$-sequence then:
\begin{enumerate}[{\em (i)}]
\item $\mathsf{D}[f_\bullet]$ is a $\mathsf{D}$-sequence; 
\item $\mathsf{T}(f_\bullet)$ is a $\mathsf{D}$-sequence. 
\end{enumerate}
and also the following equalities hold: 
\begin{enumerate}[{\em (i)}]
\setcounter{enumi}{2}
\item $\langle 1, 0 \rangle \cdot \mathsf{T}(f_\bullet) = f_\bullet \cdot \langle 1, 0 \rangle$; 
\item $(1 \times \pi_i) \cdot \mathsf{T}(f_\bullet) =  \mathsf{T}_2(f_\bullet) \cdot (1 \times \pi_i)$ where $\mathsf{T}_2(f_\bullet) := \left \langle \pi_0 \cdot f, \left \langle (1 \times \pi_0) \cdot \mathsf{D}[f_\bullet], (1 \times \pi_1) \cdot \mathsf{D}[f_\bullet] \right \rangle \right \rangle $; 
\item $\left (1 \times (\pi_0 + \pi_1) \right)  \cdot  \mathsf{T}(f_\bullet) = \mathsf{T}_2(f_\bullet) \cdot \left (1 \times (\pi_0 + \pi_1) \right) $;
\item $\ell \cdot \mathsf{T}^{2}(f_\bullet) =  \mathsf{T}(f_\bullet) \cdot \ell$;
\item $c \cdot  \mathsf{T}^{2}(f_\bullet) =  \mathsf{T}^{2}(f_\bullet) \cdot c$. 
\end{enumerate}
\end{proposition} 
\begin{proof} Suppose that $f_\bullet$ is a $\mathsf{D}$-sequence: 
\begin{enumerate}[{\em (i)}]
\item Automatic by the definition of a $\mathsf{D}$-sequence.  
\item By Lemma \ref{Dlemma1} (iii), since $\pi_0$ is additive, $\pi_0 \cdot f_\bullet$ is a $\mathsf{D}$-sequence. By (i) and Lemma \ref{Dlemma1} (vi), it then follows that $\langle \pi_0 \cdot f_\bullet, \mathsf{D}[f_\bullet] \rangle$ is also a $\mathsf{D}$-sequence. Therefore since $\mathsf{T}(f_\bullet) = \langle \pi_0 \cdot f_\bullet, \mathsf{D}[f_\bullet] \rangle$ (Proposition \ref{Diff1} (i)), $\mathsf{T}(f_\bullet)$ is a $\mathsf{D}$-sequence. 
\item Here we use Lemma \ref{astprop2} (i), Lemma \ref{astprop3} (vii), and \textbf{[DS.1]} at $n=0$: 
\begin{align*}
\langle 1, 0 \rangle \cdot \mathsf{T}(f_\bullet) = \langle 1, 0 \rangle \cdot \langle \pi_0 \cdot f_\bullet, \mathsf{D}[f_\bullet] \rangle = \left \langle \left (\langle 1, 0 \rangle\pi_0 \right) \cdot f_\bullet, \langle 1, 0 \rangle \cdot \mathsf{D}[f_\bullet] \right \rangle = \langle f_\bullet , 0_\bullet \rangle =  f_\bullet \cdot \langle 1, 0 \rangle
\end{align*}
\item Here we use Lemma \ref{astprop2} (i) and (vi):    
\begin{align*}
\mathsf{T}_2(f_\bullet) \cdot (1 \times \pi_i) &=~ \left \langle \pi_0 \cdot f, \left \langle (1 \times \pi_0) \cdot \mathsf{D}[f_\bullet], (1 \times \pi_1) \cdot \mathsf{D}[f_\bullet] \right \rangle \right \rangle  \cdot (1 \times \pi_i) \\
&=~ \left \langle \pi_0 \cdot f_\bullet, \left \langle (1 \times \pi_0) \cdot \mathsf{D}[f_\bullet], (1 \times \pi_1) \cdot \mathsf{D}[f_\bullet] \right \rangle \cdot \pi_i  \right \rangle \\
&=~ \langle \pi_0 \cdot f_\bullet, (1 \times \pi_i) \cdot \mathsf{D}[f_\bullet] \rangle \\
&=~ (1 \times \pi_i) \cdot \langle \pi_0 \cdot f_\bullet, \mathsf{D}[f_\bullet] \rangle \\
&=~ (1 \times \pi_i) \cdot  \mathsf{T}(f_\bullet)  
\end{align*}
\item Here we use Lemma \ref{astprop2} (iii), Lemma \ref{astprop3} (viii), and \textbf{[DS.2]} at $n=0$:
\begin{align*}
(1 \times (\pi_0 + \pi_1))  \cdot  \mathsf{T}(f_\bullet) &=~Ê(1 \times (\pi_0 + \pi_1))  \cdot \langle \pi_0 \cdot f_\bullet, \mathsf{D}[f_\bullet] \rangle \\
&=~ \left \langle \left((1 \times (\pi_0 + \pi_1))  \pi_0 \right) \cdot f_\bullet, (1 \times (\pi_0 + \pi_1))  \cdot \mathsf{D}[f_\bullet] \right \rangle \\
&=~ \left \langle \pi_0 \cdot f_\bullet,  \left((1 \times \pi_0) \cdot \mathsf{D}[f_\bullet] \right) + \left((1 \times \pi_1) \cdot \mathsf{D}[f_\bullet] \right) \right \rangle \\
&=~ \left \langle \pi_0 \cdot f, \left \langle (1 \times \pi_0) \cdot \mathsf{D}[f_\bullet], (1 \times \pi_1) \cdot \mathsf{D}[f_\bullet] \right \rangle \right \rangle \cdot \left (1 \times (\pi_0 + \pi_1) \right)  \\
&=~ \mathsf{T}_2(f_\bullet) \cdot \left (1 \times (\pi_0 + \pi_1) \right) 
\end{align*}
\item Here we use Lemma \ref{astprop2} (i), Lemma \ref{astprop3} (ix), and \textbf{[DS.3]} at $n=0$:
\begin{align*}
\ell \cdot \mathsf{T}^2(f_\bullet) &=~Ê\ell \cdot \left \langle \pi_0 \cdot \mathsf{T}(f_\bullet), \mathsf{D}[\mathsf{T}(f_\bullet)] \right \rangle \\
&=~Ê\ell \cdot \left \langle \pi_0 \cdot \langle \pi_0 \cdot f_\bullet, \mathsf{D}[f_\bullet] \rangle, \mathsf{D}[\langle \pi_0 \cdot f_\bullet, \mathsf{D}[f_\bullet] \rangle] \right \rangle \\
&=~Ê\ell \cdot \left \langle \pi_0 \cdot \langle \pi_0 \cdot f_\bullet, \mathsf{D}[f_\bullet] \rangle, \langle \mathsf{D}[\pi_0 \cdot f_\bullet], \mathsf{D}^2[f_\bullet] \rangle \right \rangle \\
&=~Ê\ell \cdot \left \langle \pi_0 \cdot \langle \pi_0 \cdot f_\bullet, \mathsf{D}[f_\bullet] \rangle, \langle \mathsf{P}(\pi_0) \cdot \mathsf{D}[f_\bullet], \mathsf{D}^2[f_\bullet] \rangle \right \rangle \\
&=~Ê\left \langle (\ell \pi_0) \cdot \langle \pi_0 \cdot f_\bullet, \mathsf{D}[f_\bullet] \rangle, \ell \cdot  \langle \mathsf{P}(\pi_0) \cdot \mathsf{D}[f_\bullet], \mathsf{D}^2[f_\bullet] \rangle \right \rangle \\
&=~Ê\left \langle (\pi_0 \langle 1, 0 \rangle) \cdot  \left \langle \pi_0 \cdot f_\bullet, \mathsf{D}[f_\bullet]  \right \rangle,   \left \langle (\ell \mathsf{P}(\pi_0)) \cdot \mathsf{D}[f_\bullet], \ell \cdot \mathsf{D}^2[f_\bullet] \right \rangle \right \rangle \\
&=~Ê\left \langle \pi_0 \cdot \left \langle (\langle 1, 0 \rangle \pi_0) \cdot f_\bullet, \langle 1, 0 \rangle \cdot \mathsf{D}[f_\bullet] \right \rangle,   \left \langle (\pi_0 \langle 1, 0 \rangle) \cdot \mathsf{D}[f_\bullet], \mathsf{D}[f_\bullet] \right \rangle \right \rangle \\
&=~ \left \langle \pi_0 \cdot  \langle 1 \cdot f_\bullet, 0_\bullet \rangle,   \langle \pi_0 \cdot (\langle 1, 0 \rangle \cdot \mathsf{D}[f_\bullet]) , \mathsf{D}[f_\bullet] \rangle \right \rangle \\
&=~Ê\left \langle  \langle \pi_0 \cdot  f_\bullet,  \pi_0 \cdot  0_\bullet \rangle,   \langle \pi_0 \cdot 0_\bullet , \mathsf{D}[f_\bullet] \rangle \right \rangle \\
&=~Ê\left \langle   \langle \pi_0 \cdot f_\bullet, 0_\bullet \rangle,   \langle 0_\bullet , \mathsf{D}[f_\bullet] \rangle \right \rangle \\
&=~ \langle \pi_0 \cdot f_\bullet, \mathsf{D}[f_\bullet] \rangle \cdot \ell \\
&=~ \mathsf{T}(f_\bullet) \cdot \ell
\end{align*}
\item Here we use Lemma \ref{astprop2} (i) and (viii), and \textbf{[DS.4]} at $n=0$:
\begin{align*}
c \cdot \mathsf{T}^2(f_\bullet) &=~ c \cdot \left \langle \pi_0 \cdot \langle \pi_0 \cdot f_\bullet, \mathsf{D}[f_\bullet] \rangle, \langle \mathsf{P}(\pi_0) \cdot \mathsf{D}[f_\bullet], \mathsf{D}^2[f_\bullet] \rangle \right \rangle \\
&=~ \left \langle (c \pi_0) \cdot \langle \pi_0 \cdot f_\bullet, \mathsf{D}[f_\bullet] \rangle, c \cdot \langle \mathsf{P}(\pi_0) \cdot \mathsf{D}[f_\bullet], \mathsf{D}^2[f_\bullet] \rangle \right \rangle \\
&=~ \left \langle\mathsf{P}(\pi_0) \cdot \left \langle \pi_0 \cdot f_\bullet, \mathsf{D}[f_\bullet] \right \rangle,  \left \langle (c \mathsf{P}(\pi_0)) \cdot \mathsf{D}[f_\bullet], c \cdot \mathsf{D}^2[f_\bullet] \right \rangle \right \rangle \\
&=~ \left \langle \left \langle (\mathsf{P}(\pi_0) \pi_0) \cdot f_\bullet, \mathsf{P}(\pi_0) \cdot \mathsf{D}[f_\bullet] \right \rangle,  \left \langle \pi_0 \cdot \mathsf{D}[f_\bullet], \mathsf{D}^2[f_\bullet] \right \rangle \right \rangle \\
&=~ \left  \langle \left \langle (\pi_0 \pi_0) \cdot f_\bullet, \mathsf{P}(\pi_0) \cdot \mathsf{D}[f_\bullet] \right \rangle,  \left \langle \pi_0 \cdot \mathsf{D}[f_\bullet], \mathsf{D}^2[f_\bullet] \right \rangle \right \rangle \\
&=~ \left \langle \left \langle (\pi_0 \pi_0) \cdot f_\bullet, \pi_0 \cdot \mathsf{D}[f_\bullet] \right \rangle,  \left \langle  \mathsf{P}(\pi_0) \cdot \mathsf{D}[f_\bullet], \mathsf{D}^2[f_\bullet] \right \rangle \right \rangle \cdot c \\
&=~ \left \langle \pi_0 \cdot \left \langle \pi_0 \cdot f_\bullet, \mathsf{D}[f_\bullet] \right \rangle,  \left \langle  \mathsf{P}(\pi_0) \cdot \mathsf{D}[f_\bullet], \mathsf{D}^2[f_\bullet] \right \rangle \right \rangle \cdot c \\
&=~ \mathsf{T}^2(f_\bullet) \cdot c
\end{align*}
\end{enumerate}
\end{proof} 

\begin{lemma}\label{Dcomp} If $f_\bullet$ and $g_\bullet$ are $\mathsf{D}$-sequences then $f_\bullet \ast g_\bullet$ is a $\mathsf{D}$-sequence. 
\end{lemma}
\begin{proof} We will show {\bf [DS.1]} to {\bf [DS.4]} using the identities from Proposition \ref{Dprop1}, and in particular the higher-order version of Proposition \ref{Diff1} (iv): 
\begin{enumerate}[{\bf [DS.1]}]
\item Here we use Lemma \ref{astprop2} (i), Lemma \ref{astprop1} (ii), Proposition \ref{Diff1} (iv), Proposition \ref{Dprop1} (iii), and {\bf [DS.1]} for $g_\bullet$: 
\begin{align*}
\langle 1, 0 \rangle \cdot \mathsf{D}^{n+1}[f_\bullet \ast g_\bullet] &=~ \langle 1, 0 \rangle \cdot (\mathsf{T}^{n+1}(f_\bullet) \ast \mathsf{D}^{n+1}[g_\bullet]) \\
&=~ (\langle 1, 0 \rangle \cdot \mathsf{T}^{n+1}(f_\bullet)) \ast \mathsf{D}^{n+1}[g_\bullet] \\
&=~ (\mathsf{T}^n(f_\bullet) \cdot \langle 1, 0 \rangle) \ast \mathsf{D}^{n+1}[g_\bullet] \\
&=~ \mathsf{T}^n(f_\bullet) \ast (\langle 1, 0 \rangle \cdot \mathsf{D}^{n+1}[g_\bullet]) \\
&=~  \mathsf{T}^n(f_\bullet) \ast 0_\bulletÊ\\
&=~ 0_\bullet
\end{align*}
\item Here we use Lemma \ref{astprop2} (i), Lemma \ref{astprop1} (ii), Proposition \ref{Diff1} (iv), Proposition \ref{Dprop1} (iv) and (v), the additive structure, and {\bf [DS.2]} for $g_\bullet$: 
\begin{align*}
&\left(1 \times (\pi_0 + \pi_1) \right) \cdot \mathsf{D}^{n+1}[f_\bullet \ast g_\bullet] =~ \left(1 \times (\pi_0 + \pi_1)\right) \cdot (\mathsf{T}^{n+1}(f_\bullet) \ast \mathsf{D}^{n+1}[g_\bullet]) \\
&=~ \left(\left(1 \times (\pi_0 + \pi_1)\right) \cdot \mathsf{T}^{n+1}(f_\bullet) \right) \ast \mathsf{D}^{n+1}[g_\bullet] \\
&=~ \left(\mathsf{T}_2(\mathsf{T}^{n}(f_\bullet)) \cdot \left(1 \times (\pi_0 + \pi_1)\right) \right) \ast \mathsf{D}^{n+1}[g_\bullet] \\
&=~ \mathsf{T}_2(\mathsf{T}^{n}(f_\bullet)) \ast \left( \left(1 \times (\pi_0 + \pi_1)\right) \cdot \mathsf{D}^{n+1}[g_\bullet] \right) \\
&=~Ê\mathsf{T}_2(\mathsf{T}^{n}(f_\bullet)) \ast \left( \left((1 \times \pi_1) \cdot \mathsf{D}^{n+1}[g_\bullet] \right) + \left((1 \times \pi_1) \cdot \mathsf{D}^{n+1}[g_\bullet] \right) \right)  \\
&=~ \left(Ê\mathsf{T}_2(\mathsf{T}^{n}(f_\bullet)) \ast \left((1 \times \pi_1) \cdot \mathsf{D}^{n+1}[g_\bullet] \right) \right)+ \left( \mathsf{T}_2(\mathsf{T}^{n}(f_\bullet)) \ast \left((1 \times \pi_1) \cdot \mathsf{D}^{n+1}[g_\bullet] \right) \right)  \\
&=~ \left(Ê\left( \mathsf{T}_2(\mathsf{T}^{n}(f_\bullet)) \cdot (1 \times \pi_0) \right) \ast \mathsf{D}^{n+1}[g_\bullet]  \right)+ \left(Ê\left( \mathsf{T}_2(\mathsf{T}^{n}(f_\bullet)) \cdot (1 \times \pi_1) \right) \ast \mathsf{D}^{n+1}[g_\bullet]  \right) \\
&=~ \left(Ê\left( (1 \times \pi_0) \cdot \mathsf{T}^{n+1}(f_\bullet) \right) \ast \mathsf{D}^{n+1}[g_\bullet]  \right)+ \left(Ê\left( (1 \times \pi_1) \cdot \mathsf{T}^{n+1}(f_\bullet) \right) \ast \mathsf{D}^{n+1}[g_\bullet]  \right) \\
&=~ \left(Ê(1 \times \pi_0) \cdot (\mathsf{T}^{n+1}(f_\bullet) \ast \mathsf{D}^{n+1}[g_\bullet])   \right)+  \left(Ê(1 \times \pi_1) \cdot (\mathsf{T}^{n+1}(f_\bullet) \ast \mathsf{D}^{n+1}[g_\bullet])   \right) \\
&=~ \left(Ê(1 \times \pi_0) \cdot (\mathsf{T}^{n+1}(f_\bullet) \ast \mathsf{D}^{n+1}[g_\bullet])   \right)+  \left(Ê(1 \times \pi_1) \cdot (\mathsf{T}^{n+1}(f_\bullet) \ast \mathsf{D}^{n+1}[g_\bullet])   \right) \\
&=~ \left((1 \times \pi_1) \cdot \mathsf{D}^{n+1}[f_\bullet \ast g_\bullet] \right) + \left((1 \times \pi_1) \cdot \mathsf{D}^{n+1}[f_\bullet \ast g_\bullet] \right)
\end{align*}
\item Here we use Lemma \ref{astprop2} (i), Lemma \ref{astprop1} (ii), Proposition \ref{Diff1} (iv), Proposition \ref{Dprop1} (vi), and {\bf [DS.3]} for $g_\bullet$: 
\begin{align*}
\ell \cdot \mathsf{D}^{n+2}[f_\bullet \ast g_\bullet] &=~ \ell \cdot (\mathsf{T}^{n+2}(f_\bullet) \ast \mathsf{D}^{n+2}[g_\bullet]) \\
&=~ (\ell \cdot \mathsf{T}^{n+2}(f_\bullet)) \ast \mathsf{D}^{n+2}[g_\bullet] \\
&=~ (\mathsf{T}^{n+1}(f_\bullet) \cdot \ell) \ast \mathsf{D}^{n+2}[g_\bullet] \\
&=~ \mathsf{T}^{n+1}(f_\bullet) \ast (\ell \cdot \mathsf{D}^{n+2}[g_\bullet]) \\
&=~  \mathsf{T}^{n+1}(f_\bullet) \ast \mathsf{D}^{n+1}[g_\bullet]Ê\\
&=~ \mathsf{D}^{n+1}[f_\bullet \ast g_\bullet] 
\end{align*}
\item Here we use Lemma \ref{astprop2} (i), Lemma \ref{astprop1} (ii), Proposition \ref{Diff1} (iv), Proposition \ref{Dprop1} (vii), and {\bf [DS.4]} for $g_\bullet$: 
\begin{align*}
c \cdot \mathsf{D}^{n+2}[f_\bullet \ast g_\bullet] &=~ c \cdot (\mathsf{T}^{n+2}(f_\bullet) \ast \mathsf{D}^{n+2}[g_\bullet]) \\
&=~ (c \cdot \mathsf{T}^{n+2}(f_\bullet)) \ast \mathsf{D}^{n+2}[g_\bullet] \\
&=~ (\mathsf{T}^{n+2}(f_\bullet) \cdot c) \ast \mathsf{D}^{n+2}[g_\bullet] \\
&=~ \mathsf{T}^{n+2}(f_\bullet) \ast (c \cdot \mathsf{D}^{n+2}[g_\bullet]) \\
&=~  \mathsf{T}^{n+2}(f_\bullet) \ast \mathsf{D}^{n+2}[g_\bullet]Ê\\
&=~ \mathsf{D}^{n+2}[f_\bullet \ast g_\bullet] 
\end{align*}
\end{enumerate}
\end{proof} 

Finally we may properly state that we obtain a category of $\mathsf{D}$-sequences: 

\begin{definition}\label{Dcat} Let $\mathbb{X}$ be a Cartesian left additive category. Then we denote $\mathcal{D}[\mathbb{X}]$ as the sub-Cartesian left additive category of $\overline{\mathcal{D}}[\mathbb{X}]$ of $\mathsf{D}$-sequences of $\mathbb{X}$. \end{definition}

\begin{proposition}\label{Dseqcat} $\mathcal{D}[\mathbb{X}]$ is a Cartesian left additive category. 
\end{proposition} 
\begin{proof} Composition is well defined by Lemma \ref{Dcomp}, and the identities are well defined by Lemma \ref{Dlemma1} (ii). The left additive structure is well defined by Lemma \ref{Dlemma1} (i) and (vii). The finite product structure is well defined by Lemma \ref{Dlemma1} (v) and (vi). Then that $\mathcal{D}[\mathbb{X}]$ is a Cartesian left additive category follows from being a subcategory of $\overline{\mathcal{D}}[\mathbb{X}]$. 
\end{proof} 

\subsection{Comonad of $\mathsf{D}$-Sequences}\label{Dcomsec}

In this section we show that the category of $\mathsf{D}$-sequences does indeed provide a comonad on the category of Cartesian left additive categories $\mathsf{CLAC}$ (as defined in Section \ref{predaddsec}). In particular, as we will show in the next section, the coalgebras of this comonad are precisely Cartesian differential categories. While most of the work in showing that we have a comonad was done in Section \ref{predcomsec}, we still need to show that the $\mathsf{D}$-sequence axioms are well preserved. 

Let $\mathsf{F}: \mathbb{X} \to \mathbb{Y}$ be a strict Cartesian left additive functor, then define the functor $\mathcal{D}[\mathsf{F}]: \mathcal{D}[\mathbb{X}] \to \mathcal{D}[\mathbb{Y}]$ to be the restriction of $\overline{\mathcal{D}}[\mathsf{F}]$ (as defined in Lemma \ref{DFunctor}) to the category of $\mathsf{D}$-sequences. Explicitly, on objects $\mathcal{D}[\mathsf{F}](A)=\mathsf{F}(A)=\overline{\mathcal{D}}[\mathsf{F}](A)$, while on $\mathsf{D}$-sequences, $\mathcal{D}[\mathsf{F}](f_\bullet)=\overline{\mathcal{D}}[\mathsf{F}](f_\bullet)$.

\begin{lemma}\label{DFunctor2} If $\mathsf{F}: \mathbb{X} \to \mathbb{Y}$ is a strict Cartesian left additive functor, then $\mathcal{D}[\mathsf{F}]: \mathcal{D}[\mathbb{X}] \to \mathcal{D}[\mathbb{Y}]$ is a strict Cartesian left additive functor. 
\end{lemma}
\begin{proof} By Lemma \ref{Daddfunc}, we only need to check that when $f_\bullet$ is a $\mathsf{D}$-sequence, then so is $\mathcal{D}[\mathsf{F}](f_\bullet)$. Note that since $\mathsf{F}$ is a strict Cartesian left additive functor, we have that: 
\begin{align*}
\mathsf{F}(\langle 1, 0 \rangle) = \langle 1, 0 \rangle && \mathsf{F}(1 \times (\pi_0 + \pi_1)) = 1 \times (\pi_0 + \pi_1) && \mathsf{F}(1 \times \pi_j)=1 \times \pi_j \\
\mathsf{F}(\ell) = \ell && \mathsf{F}(c)=c &&
\end{align*}
And recall that $\mathsf{F} \mathsf{P} = \mathsf{P} \mathsf{F}$. Therefore it is easier to check that $\mathcal{D}[\mathsf{F}](f_\bullet)$ satisfies {\bf [DS.$1^\prime$]} to {\bf [DS.$4^\prime$]}. In the following, let $k \leq n$:  \\Ê\\
{\bf [DS.$1^\prime$]} Here we use {\bf [DS.$1^\prime$]} for $f_\bullet$: 
\begin{align*}
\mathsf{P}^k(\langle 1, 0 \rangle) \mathcal{D}[\mathsf{F}](f_\bullet)_{n+1} &=~ \mathsf{P}^k(\mathsf{F}(\langle 1, 0 \rangle)) \mathsf{F}(f_{n+1})\\
&=~  \mathsf{F}(\mathsf{P}^k(\langle 1, 0 \rangle)) \mathsf{F}(f_{n+1}) \\
&=~  \mathsf{F}(\mathsf{P}^k(\langle 1, 0 \rangle) f_{n+1})\\
&=~ \mathsf{F}(0) \\
&=~ 0
\end{align*}
{\bf [DS.$2^\prime$]} Here we use {\bf [DS.$2^\prime$]} for $f_\bullet$: 
\begin{align*}
\mathsf{P}^k(1 \times (\pi_0 + \pi_1)) \mathcal{D}[\mathsf{F}](f_\bullet)_{n+1} &=~Ê\mathsf{P}^k\left(\mathsf{F}(1 \times (\pi_0 + \pi_1))\right)  \mathsf{F}(f_{n+1})\\
&=~ \mathsf{F}\left(\mathsf{P}^k(1 \times (\pi_0 + \pi_1))\right) \mathsf{F}(f_{n+1})\\
&=~ \mathsf{F}\left(\mathsf{P}^k(1 \times (\pi_0 + \pi_1)) f_{n+1} \right) \\
&=~ \mathsf{F}\left(\mathsf{P}^k(1 \times \pi_0)f_{n+1} + \mathsf{P}^k(1\times \pi_1)f_{n+1} \right)\\
&=~ \mathsf{F}\left(\mathsf{P}^k(1 \times \pi_0)f_{n+1}\right) + \mathsf{F}\left(\mathsf{P}^k(1 \times \pi_1)f_{n+1}\right) \\
&=~ \mathsf{F}\left(\mathsf{P}^k(1 \times \pi_0) \right) \mathsf{F}(f_{n+1}) +  \mathsf{F}\left(\mathsf{P}^k(1 \times \pi_1) \right) \mathsf{F}(f_{n+1})\\
&=~ Ê\mathsf{P}^k\left(\mathsf{F}(1 \times \pi_0 )\right)  \mathsf{F}(f_{n+1}) + \mathsf{P}^k\left(\mathsf{F}(1 \times \pi_1 )\right)  \mathsf{F}(f_{n+1}) \\
&=~ \mathsf{P}^k(1 \times \pi_0) \mathcal{D}[\mathsf{F}](f_\bullet)_{n+1} + \mathsf{P}^k(1 \times \pi_1) \mathcal{D}[\mathsf{F}](f_\bullet)_{n+1} 
\end{align*}
{\bf [DS.$3^\prime$]} Here we use {\bf [DS.$3^\prime$]} for $f_\bullet$: 
\begin{align*}
\mathsf{P}^k(\ell) \mathcal{D}[\mathsf{F}](f_\bullet)_{n+2}&=~ \mathsf{P}^k(\mathsf{F}(\ell)) \mathsf{F}(f_{n+2})\\
&=~ \mathsf{F}(\mathsf{P}^k(\ell))\mathsf{F}(f_{n+2})\\
&=~ \mathsf{F}(\mathsf{P}^k(\ell)f_{n+2}) \\
&=~ \mathsf{F}(f_{n+1})\\
&=~\mathcal{D}[\mathsf{F}](f_\bullet)_{n+1}
\end{align*}
{\bf [DS.$4^\prime$]} Here we use {\bf [DS.$4^\prime$]} for $f_\bullet$: 
\begin{align*}
\mathsf{P}^k(c) \mathcal{D}[\mathsf{F}](f_\bullet)_{n+2}&=~ \mathsf{P}^k(\mathsf{F}(c)) \mathsf{F}(f_{n+2})\\
&=~ \mathsf{F}(\mathsf{P}^k(c))\mathsf{F}(f_{n+2})\\
&=~ \mathsf{F}(\mathsf{P}^k(c)f_{n+2}) \\
&=~ \mathsf{F}(f_{n+2})\\
&=~ \mathcal{D}[\mathsf{F}](f_\bullet)_{n+2}
\end{align*}
\end{proof} 

\begin{lemma} $\mathcal{D}: \mathsf{CLAC} \to \mathsf{CLAC}$ is a functor. 
\end{lemma}
\begin{proof} That $\mathcal{D}$ is well defined on objects (Cartesian left additive categories) follows from Proposition \ref{Dseqcat}, while Lemma \ref{DFunctor2} says that $\mathcal{D}$ is well defined on maps (strict Cartesian left additive functors). That $\mathcal{D}$ preserves identities and composition follows from Lemma \ref{Dbarfunctor}.
\end{proof} 

The comonad structure on $\mathcal{D}$ is precisely the same as the comonad structure on $\overline{\mathcal{D}}$, that is,  the counit is defined as $\varepsilon := \overline{\varepsilon}$ (as defined in Lemma \ref{epsilonlemma1}) and the comultiplication is defined as $\delta := \overline{\delta}$ (as defined in Lemma \ref{deltalemma1}). We still have to check however that $\delta$ is well defined.  

\begin{lemma}\label{epsilonlemma3} $\varepsilon: \mathcal{D}[\mathbb{X}] \to \mathbb{X}$ is a strict Cartesian left additive functor. 
\end{lemma}
\begin{proof} Follows immediately from Lemma \ref{epsilonadd}. 
\end{proof} 

To check that $\delta$ indeed maps $\mathsf{D}$-sequences to $\mathsf{D}$-sequences, we will need the following useful identities (which are straightforward to check): 

\begin{lemma}\label{ptanprop} The following equalities hold: 
\begin{enumerate}[{\em (i)}]
\item $\langle i_\bullet, 0_\bullet \rangle = i_\bullet \cdot \langle 1, 0 \rangle$;
\item $(i_\bullet \times (i_\bullet \cdot \pi_0) \times (i_\bullet \cdot \pi_1)) = i_\bullet \cdot (1 \times (\pi_0 + \pi_1))$; 
\item $\langle i_\bullet, 0_\bullet \rangle \times \langle 0_\bullet, i_\bullet \rangle = i_\bullet \cdot \ell$; 
\item $i_\bullet \times \langle i_\bullet \cdot \pi_1, i_\bullet \cdot \pi_0 \rangle \times i_\bullet = i_\bullet \cdot c$;
\end{enumerate}
\end{lemma} 

\begin{lemma}\label{deltalemma3} $\delta: \mathcal{D}[\mathbb{X}] \to \mathcal{D}\left[\mathcal{D}[\mathbb{X}] \right]$ is a strict Cartesian left additive functor. 
\end{lemma}
\begin{proof} We will use the same notation for $\mathcal{D}\left[\mathcal{D}[\mathbb{X}] \right]$ as introduced in the proof of Lemma \ref{deltalemma1}. Using that $\mathsf{P}(i_\bullet \cdot k)=\mathsf{T}(i_\bullet \cdot k)$ (as shown in the proof of Lemma \ref{deltalemma1}) and the higher order version of Proposition \ref{Diff1} (iv), we have that: 
\begin{align*}
\left( (i_\bullet \cdot k) \odot \mathsf{D}^{n+q}[\delta(f_\bullet)] \right)_m &=~ \mathsf{P}^m(i_\bullet \cdot k) \ast \mathsf{D}^{n+q}[\delta(f_\bullet)]_m \\
&=~ \mathsf{P}^m(i_\bullet \cdot k) \ast \delta(f_\bullet)_{n+m+q} \\
&=~  \mathsf{P}^m(i_\bullet \cdot k) \ast \mathsf{D}^{n+m+q}[f_\bullet] \\
&=~ \mathsf{D}^{m}\left[(i_\bullet \cdot k) \ast \mathsf{D}^{n+q}[f] \right] \\
&=~ \mathsf{D}^{m}\left[k \cdot \mathsf{D}^{n+q}[f] \right] 
\end{align*}
Using the identities of Lemma \ref{ptanprop} and this above identity, we check {\bf [DS.1]} to {\bf [DS.4]}. \\Ê\\
{\bf [DS.1]} Here we use Lemma \ref{ptanprop} (i) and {\bf [DS.1]} for $f_\bullet$: 
\begin{align*}
\left( \langle i_\bullet, 0_\bullet \rangle \odot \mathsf{D}^{n+1}[\delta(f_\bullet)] \right)_m &=~ \left( (i_\bullet \cdot \langle 1, 0 \rangle) \odot \mathsf{D}^{n+1}[\delta(f_\bullet)] \right)_m\\
&=~ \mathsf{D}^{m}\left[\langle 1, 0 \rangle \cdot \mathsf{D}^{n+1}[f] \right] \\
&=~ \mathsf{D}^m[0_\bullet] \\
&=~ 0_\bullet 
\end{align*}
{\bf [DS.2]} Here we use Lemma \ref{ptanprop} (ii) and {\bf [DS.2]} for $f_\bullet$: 
\begin{align*}
&\left( \left( i_\bullet \times \left( (i_\bullet \cdot \pi_0) + (i_\bullet \cdot \pi_1) \right) \right) \odot \mathsf{D}^{n+1}[\delta(f_\bullet)] \right)_m =~Ê\left( \left( i_\bullet \cdot (1 \times (\pi_0 + \pi_1)) \right) \odot \mathsf{D}^{n+1}[\delta(f_\bullet)] \right)_m \\
&=~  \mathsf{D}^{m}\left[(1 \times (\pi_0 + \pi_1)) \cdot \mathsf{D}^{n+1}[f_\bullet] \right] \\
&=~ \mathsf{D}^m\left[ (1 \times \pi_0) \cdot \mathsf{D}^{n+1}[f_\bullet] + (1 \times \pi_1) \cdot \mathsf{D}^{n+1}[f_\bullet] \right] \\
&=~ \mathsf{D}^m\left[ (1 \times \pi_0) \cdot \mathsf{D}^{n+1}[f_\bullet] \right] + \mathsf{D}^m\left[ (1 \times \pi_1) \cdot \mathsf{D}^{n+1}[f_\bullet] \right] \\
&=~Ê\left( \left( i_\bullet \cdot (1 \times \pi_0) \right) \odot \mathsf{D}^{n+1}[\delta(f_\bullet)] \right)_m + \left( \left( i_\bullet \cdot (1 \times \pi_1) \right) \odot \mathsf{D}^{n+1}[\delta(f_\bullet)] \right)_m\\
&=~\left( \left( i_\bullet \times (i_\bullet \cdot \pi_0) \right)  \odot \mathsf{D}^{n+1}[\delta(f_\bullet)] \right)_m + \left( \left( i_\bullet \times (i_\bullet \cdot \pi_1) \right)  \odot \mathsf{D}^{n+1}[\delta(f_\bullet)] \right)_m \\
&=~Ê\left( \left( i_\bullet \times (i_\bullet \cdot \pi_0) \right)  \odot \mathsf{D}^{n+1}[\delta(f_\bullet)]  + \left( i_\bullet \times (i_\bullet \cdot \pi_1) \right)  \odot \mathsf{D}^{n+1}[\delta(f_\bullet)] \right)_m
\end{align*}
{\bf [DS.3]} Here we use Lemma \ref{ptanprop} (iii) and {\bf [DS.3]} for $f_\bullet$: 
\begin{align*}
\left( \left( \langle i_\bullet, 0_\bullet \rangle \times \langle 0_\bullet, i_\bullet \rangle \right) \odot \mathsf{D}^{n+2}[\delta(f_\bullet)] \right)_m &=~ \left( (i_\bullet \cdot \ell) \odot \mathsf{D}^{n+2}[\delta(f_\bullet)] \right)_m \\
&=~ \mathsf{D}^{m}\left[\ell \cdot \mathsf{D}^{n+2}[f_\bullet] \right] \\
&=~ \mathsf{D}^m[\mathsf{D}^{n+1}[f_\bullet] ] \\
&=~ \mathsf{D}^{n+m+1}[f_\bullet] \\
&=~ \delta(f_\bullet)_{n+m+1} \\
&=~ \left(\mathsf{D}^{n+1}[\delta(f_\bullet)] \right)_m
\end{align*}
{\bf [DS.4]} Here we use Lemma \ref{ptanprop} (iv) and {\bf [DS.4]} for $f_\bullet$: 
\begin{align*}
\left( \left( i_\bullet \times \langle i_\bullet \cdot \pi_1, i_\bullet \cdot \pi_0 \rangle \times i_\bullet \right) \odot \mathsf{D}^{n+2}[\delta(f_\bullet)] \right)_m &=~ \left( (i_\bullet \cdot c) \odot \mathsf{D}^{n+2}[\delta(f_\bullet)] \right)_m \\
&=~ \mathsf{D}^{m}\left[c \cdot \mathsf{D}^{n+2}[f_\bullet] \right] \\
&=~ \mathsf{D}^m[\mathsf{D}^{n+2}[f_\bullet] ] \\
&=~ \mathsf{D}^{n+m+2}[f_\bullet] \\
&=~ \delta(f_\bullet)_{n+m+2} \\
&=~ \left(\mathsf{D}^{n+2}[\delta(f_\bullet)] \right)_m
\end{align*}
\end{proof} 

\begin{lemma} $\varepsilon: \mathcal{D} \Rightarrow 1_{\mathsf{CLAC}}$ and $\delta: \mathcal{D} \Rightarrow \mathcal{D}\mathcal{D}$ are both natural transformations. 
\end{lemma}
\begin{proof} $\varepsilon$ and $\delta$ are well defined by Lemma \ref{epsilonlemma3} and Lemma \ref{deltalemma3}, while their naturality were shown in Lemma \ref{natepsilon} and Lemma \ref{natdelta}. 
\end{proof} 

Finally we obtain the desired comonad on $\mathsf{CLAC}$: 

\begin{proposition}\label{predcom3} $(\mathcal{D}, \delta, \varepsilon)$ is a comonad on $\mathsf{CLAC}$.
\end{proposition} 
\begin{proof} That $(\mathcal{D}, \delta, \varepsilon)$ is a comonad follows immediately from Proposition \ref{predcom2}. 
\end{proof} 

\subsection{Cofree Cartesian Differential Categories}

In this section we prove the main result of this paper: that $\mathcal{D}$-coalgebras of the comonad $(\mathcal{D}, \delta, \varepsilon)$ are precisely Cartesian differential categories, or in other words, the category of $\mathcal{D}$-coalgebras is equivalent to the category of Cartesian differential categories. This then implies that for a Cartesian left additive category $\mathbb{X}$, its category of $\mathsf{D}$-sequences $\mathcal{D}[\mathbb{X}]$ is indeed the cofree Cartesian differential category over $\mathbb{X}$. 

Recall that a $\mathcal{D}$-coalgebra is a pair $(\mathbb{X}, \omega)$ consisting of a Cartesian left additive category $\mathbb{X}$ and strict Cartesian left additive functor $\omega: \mathbb{X} \to \mathcal{D}[\mathbb{X}]$ such that the following diagrams commute: 
  \[  \xymatrixcolsep{5pc}\xymatrix{\mathbb{X} \ar@{=}[dr]^-{} \ar[r]^-{\omega} & \mathcal{D}[\mathbb{X}] \ar[d]^-{\varepsilon} & \mathbb{X} \ar[d]_-{\omega} \ar[r]^-{\omega} & \mathcal{D}[\mathbb{X}] \ar[d]^-{\delta} \\
  & \mathbb{X} & \mathcal{D}[\mathbb{X}] \ar[r]_-{\mathcal{D}[\omega]} & \mathcal{D}\left[ \mathcal{D}[\mathbb{X}] \right]
  } \]
  And that a $\mathcal{D}$-coalgebra morphism $\mathsf{F}: (\mathbb{X}, \omega) \to (\mathbb{Y}, \omega^\prime)$ is a strict Cartesian left additive functor $\mathsf{F}: \mathbb{X} \to \mathbb{Y}$ such that the following diagram commutes: 
    \[  \xymatrixcolsep{5pc}\xymatrix{\mathbb{X} \ar[d]_-{\omega} \ar[r]^-{\mathsf{F}} & \mathbb{Y} \ar[d]^-{\omega^\prime} \\
     \mathcal{D}[\mathbb{X}] \ar[r]_-{\mathcal{D}[\mathsf{F}]} & \mathcal{D}[\mathbb{Y}]
  } \]
  
We start by showing that every Cartesian differential category is a  $\mathcal{D}$-coalgebra. Let $\mathbb{X}$ be a Cartesian differential category with differential combinator $\mathsf{D}$ (the author apologizes in advance for the repetitive notation). Define the functor $\omega^{\mathsf{D}}: \mathbb{X} \to  \mathcal{D}[\mathbb{X}]$ on objects as $\omega^{\mathsf{D}}(A) := A$ and for maps $f$, $\omega^{\mathsf{D}}(f)_\bullet$ is the $\mathsf{D}$-sequence defined as $\omega^{\mathsf{D}}(f)_n := \mathsf{D}^n[f]$ for $n \geq 1$ and $\omega^{\mathsf{D}}(f)_0 = f$. Note that $\omega^{\mathsf{D}}(f)_\bullet$ is precisely the intuition we gave for $\mathsf{D}$-sequences. 

\begin{lemma}\label{omegalemma} $\omega^{\mathsf{D}}(f)_\bullet$ is a $\mathsf{D}$-sequence of $\mathbb{X}$.
\end{lemma}
\begin{proof} The proof follows the same arguments that was provided when giving intuition for the $\mathsf{D}$-sequence axioms {\bf [DS.$1^\prime$]} to {\bf [DS.$4^\prime$]}. The key is that $\langle 1, 0 \rangle$, $\left (1 \times (\pi_0 + \pi_1) \right)$, $1 \times \pi_i$, $\ell$, and $c$ are all linear in the Cartesian differential category sense. And recall that for a linear map $h$, by Proposition \ref{tangentfunctorprop}, we have that $\mathsf{T}(h)= h \times h = \mathsf{P}(h)$. Then using the higher-order chain rule that $\mathsf{D}^n[fg] = \mathsf{T}^n(f) \mathsf{D}^n[g]$, we can easily check {\bf [DS.$1^\prime$]} to {\bf [DS.$4^\prime$]}: \\Ê\\Ê
{\bf [DS.$1^\prime$]}: Here we use \textbf{[CD.2]}, that $\langle 1,0 \rangle$ is linear, and the higher-order chain rule \textbf{[CD.5]}: 
\begin{align*}
\mathsf{P}^k(\langle 1, 0 \rangle) \omega^{\mathsf{D}}(f)_{n+1} &=~ \mathsf{T}^k(\langle 1, 0 \rangle) \mathsf{D}^{n+1}[f] \\
&=~ \mathsf{D}^{k}\left[ \langle 1, 0 \rangle \mathsf{D}^{n+1-k}[f] \right] \\
&=~  \mathsf{D}^{k}\left[ \langle 1, 0 \rangle \mathsf{D}\left[\mathsf{D}^{n-k}[f] \right] \right] \\
&=~ \mathsf{D}^k[0] \\
&=~ 0
\end{align*}
{\bf [DS.$2^\prime$]}: Here we again use \textbf{[CD.2]}, that $\left (1 \times (\pi_0 + \pi_1) \right)$ and $1 \times \pi_i$ are linear, and the higher-order chain rule \textbf{[CD.5]}: 
\begin{align*}
\mathsf{P}^k\left(1 \times (\pi_0 + \pi_1) \right) \omega^{\mathsf{D}}(f)_{n+1} &=~ \mathsf{T}^k\left(1 \times (\pi_0 + \pi_1) \right) \mathsf{D}^{n+1}[f] \\
&=~  \mathsf{D}^{k}\left[ \left(1 \times (\pi_0 + \pi_1) \right) \mathsf{D}^{n+1-k}[f] \right] \\
&=~Ê \mathsf{D}^{k}\left[ \left(1 \times (\pi_0 + \pi_1) \right) \mathsf{D}\left[\mathsf{D}^{n-k}[f] \right] \right]  \\Ê
&=~Ê \mathsf{D}^{k}\left[ (1 \times \pi_0) \mathsf{D}\left[\mathsf{D}^{n-k}[f] \right] + (1 \times \pi_1) \mathsf{D}\left[\mathsf{D}^{n-k}[f] \right] \right]  \\Ê
&=~Ê \mathsf{D}^{k}\left[ (1 \times \pi_0) \mathsf{D}\left[\mathsf{D}^{n-k}[f] \right] \right] +\mathsf{D}^{k}\left[ (1 \times \pi_1) \mathsf{D}\left[\mathsf{D}^{n-k}[f] \right] \right]  \\Ê
&=~Ê  \mathsf{T}^k\left(1 \times \pi_0 \right) \mathsf{D}^{k}\left[ \mathsf{D}\left[\mathsf{D}^{n-k}[f] \right] \right] +\mathsf{T}^k\left(1 \times \pi_1 \right) \mathsf{D}^{k}\left[ \mathsf{D}\left[\mathsf{D}^{n-k}[f] \right] \right]   \\Ê
&=~Ê  \mathsf{P}^k\left(1 \times \pi_0 \right)  \omega^{\mathsf{D}}(f)_{n+1}+\mathsf{P}^k\left(1 \times \pi_1 \right)  \omega^{\mathsf{D}}(f)_{n+1}  
\end{align*}
{\bf [DS.$3^\prime$]}: Here we use \textbf{[CD.6]}, that $\ell$ is linear, and the higher-order chain rule \textbf{[CD.5]}: 
\begin{align*}
\mathsf{P}^k(\ell) \omega^{\mathsf{D}}(f)_{n+2} &=~ \mathsf{T}^k(\ell) \mathsf{D}^{n+2}[f] \\
&=~ \mathsf{D}^{k}\left[ \ell \mathsf{D}^{n+2-k}[f] \right] \\
&=~  \mathsf{D}^{k}\left[ \ell \mathsf{D}^2\left[\mathsf{D}^{n-k}[f] \right] \right] \\
&=~  \mathsf{D}^{k}\left[ \mathsf{D}\left[\mathsf{D}^{n-k}[f] \right] \right] \\
&=~ \omega^{\mathsf{D}}(f)_{n+1} 
\end{align*}
{\bf [DS.$4^\prime$]}: Here we use \textbf{[CD.7]}, that $c$ is linear, and the higher-order chain rule \textbf{[CD.5]}: 
\begin{align*}
\mathsf{P}^k(c) \omega^{\mathsf{D}}(f)_{n+2} &=~ \mathsf{T}^k(c) \mathsf{D}^{n+2}[f] \\
&=~ \mathsf{D}^{k}\left[ c \mathsf{D}^{n+2-k}[f] \right]\\
&=~  \mathsf{D}^{k}\left[ c \mathsf{D}^2\left[\mathsf{D}^{n-k}[f] \right] \right] \\
&=~  \mathsf{D}^{k}\left[ \mathsf{D}^2\left[\mathsf{D}^{n-k}[f] \right] \right]\\
&=~ \omega^{\mathsf{D}}(f)_{n+2}
\end{align*}
\end{proof} 

\begin{lemma}\label{etalemma1} $\omega^{\mathsf{D}}: \mathbb{X} \to  \mathcal{D}[\mathbb{X}]$ is a strict Cartesian left additive functor.  
\end{lemma}
\begin{proof} That $\omega^{\mathsf{D}}$ is well-defined follows from Lemma \ref{omegalemma}. Now we must check that $\omega^{\mathsf{D}}$ is indeed is a strict Cartesian left additive functor. Using the higher-order version of \textbf{[CD.3]}, we start by showing $\omega^{\mathsf{D}}$ preserves identities and projections: 
\[\omega^{\mathsf{D}}(1)_n= \mathsf{D}^n[1]= \underbrace{\pi_1 \hdots \pi_1}_{n-\text{times}} = i_n \]
\[\omega^{\mathsf{D}}(\pi_j)_n= \mathsf{D}^n[\pi_j]= \underbrace{\pi_1 \hdots \pi_1}_{n-\text{times}} \pi_j = i_n\pi_j = (i_\bullet \cdot \pi_j)_n \]
Now using the higher-order version of \textbf{[CD.1]}, we have that $\omega^{\mathsf{D}}$ preserves the additive structure: 
\[\omega^{\mathsf{D}}(0)_n= \mathsf{D}^n[0]= 0 = 0_n  \]
\[\omega^{\mathsf{D}}(f+g)_n = \mathsf{D}^n[f+g]= \mathsf{D}^n[f] + \mathsf{D}^n[g] = \omega^{\mathsf{D}}(f)_n + \omega^{\mathsf{D}}(g)_n = \left(\omega^{\mathsf{D}}(f)_\bullet + \omega^{\mathsf{D}}(g)_\bullet \right)_n\]
Lastly, to show that $\omega^{\mathsf{D}}$ preserves composition, notice that $\omega^{\mathsf{D}}$ also preserve the tangent functors immediately by definition: 
\[\mathsf{T}(\omega^{\mathsf{D}}(f)_\bullet)_0 = \langle \pi_0 \omega^{\mathsf{D}}(f)_0, \omega^{\mathsf{D}}(f)_1 \rangle = \langle \pi_0 f, \mathsf{D}[f] \rangle = \mathsf{T}(f)  \]
Now using that $\omega^{\mathsf{D}}$ preserves tangent functors and the higher-order version of \textbf{[CD.5]}, we have that: 
\begin{align*}
\omega^{\mathsf{D}}(fg)_n = \mathsf{D}^n[fg] = \mathsf{T}^n(f) \mathsf{D}^n[g] = \mathsf{T}^n(f) \mathsf{D}^n[g] = \mathsf{T}(\omega^{\mathsf{D}}(f)_\bullet)_0 \omega^{\mathsf{D}}(g)_n = (\omega^{\mathsf{D}}(f)_\bullet \ast \omega^{\mathsf{D}}(g)_\bullet)_n \end{align*}
\end{proof} 

\begin{proposition}\label{Dcoalgprop1} If $\mathbb{X}$ is a Cartesian differential category, then $(\mathbb{X}, \omega^{\mathsf{D}})$ is a $\mathcal{D}$-coalgebra. 
\end{proposition} 
\begin{proof} We must check the two diagrams of a $\mathcal{D}$-coalgebra. Though these are in fact automatic by definition. Starting with the triangle: 
\[\varepsilon(\omega^{\mathsf{D}}(f)_\bullet) = \omega^{\mathsf{D}}(f)_0 = f\]
and now the square (working again with double indexing): 
\begin{align*}
\left(\mathcal{D}[\omega^{\mathsf{D}}](\omega^{\mathsf{D}}(f)_\bullet)_n\right)_m &=~Ê\omega^{\mathsf{D}}(\omega^{\mathsf{D}}(f)_n)_m\\
&=~ \omega^{\mathsf{D}}(\mathsf{D}^n[f])_m \\
&=~ \mathsf{D}^{n+m}[f] \\
&=~Ê\omega^{\mathsf{D}}(f)_{n+m} \\
&=~Ê\mathsf{D}^n[\omega^{\mathsf{D}}(f)_\bullet]_m \\
&=~Ê\left(\delta(\omega^{\mathsf{D}}(f)_\bullet)_n\right)_m
\end{align*}
\end{proof} 

A \textbf{strict Cartesian differential functor} between Cartesian differential categories is a strict Cartesian left additive functor $\mathsf{F}$ which also preserves the differential combinator strictly in the sense that: $\mathsf{F}(\mathsf{D}[f]) = \mathsf{D}[\mathsf{F}(f)]$. We now show that strict Cartesian differential functors are in fact $\mathcal{D}$-coalgebra morphisms: 

\begin{lemma}\label{lemfunc1} If $\mathsf{F}: \mathbb{X} \to \mathbb{Y}$ is a strict Cartesian differential functor, then $\mathsf{F}: (\mathbb{X}, \omega^{\mathsf{D}}) \to (\mathbb{Y}, \omega^{\mathsf{D}})$ is a $\mathcal{D}$-coalgebra morphism. 
\end{lemma}
\begin{proof} This is again straightforward by definition: 
\[  \omega^{\mathsf{D}}(\mathsf{F}(f))_n = (\mathsf{D})^n[\mathsf{F}(f)] = \mathsf{F}(\mathsf{D}^n[f]) = \mathsf{F}(\omega^{\mathsf{D}}(f)_n) =  \mathcal{D}[\mathsf{F}](\omega^\mathsf{D}(f)_\bullet)_n \]
\end{proof} 

Now for the converse, we will show that a $\mathcal{D}$-coalgebra is a Cartesian differential category and that $\mathcal{D}$-coalgebra morphisms are strict Cartesian differential functors. 

\begin{proposition}\label{Dcoalgprop2} If $(\mathbb{X}, \omega)$ is a $\mathcal{D}$-coalgebra, then $\mathbb{X}$ is a Cartesian differential category with differential combinator defined as $\mathsf{D}^\omega[f] = \omega(f)_1$. 
\end{proposition} 
\begin{proof} We must show that $\mathsf{D}^\omega$ satisfies \textbf{[CD.1]} to \textbf{[CD.7]} (though recall that by  Lemma \ref{CD4}: \textbf{[CD.4]} is redundant). \\Ê\\
\textbf{[CD.1]}: Since $\omega$ preserves the additive structure strictly we have that $\omega(0)_\bullet = 0_\bullet$ and $\omega(f + g)_\bullet = \omega(f)_\bullet + \omega(g)_\bullet$. Therefore it follows that:
\[\mathsf{D}^\omega[0] = \omega(0)_1 = 0_1 = 0 \quad \quad \quad \mathsf{D}^\omega[f+g] = \omega(f+g)_1 = \omega(f)_1 + \omega(g)_1 = \mathsf{D}^\omega[f] + \mathsf{D}^\omega[g]\]
\textbf{[CD.2]}: Here we use {\bf [DS.$1^\prime$]} and {\bf [DS.$2^\prime$]} for $\omega(f)_\bullet$ at $n=1$: 
\[\langle 1,0 \rangle \mathsf{D}^\omega[f] = \langle 1,0 \rangle \omega(f)_1 = 0\]
\begin{align*}
\left(1 \times (\pi_0 + \pi_1) \right) \mathsf{D}^\omega[f] &=~ \left(1 \times (\pi_0 + \pi_1) \right) \omega(f)_1 \\
&=~ (1 \times \pi_0) \omega(f)_1 + (1 \times \pi_1) \omega(f)_1 \\
&=~ (1 \times \pi_0)\mathsf{D}^\omega[f] + (1 \times \pi_1)\mathsf{D}^\omega[f]
\end{align*}
\textbf{[CD.3]}: Since $\omega$ is a strict Cartesian functor, we have that $\omega(1)_\bullet=i_\bullet$ and $\omega(\pi_j)_\bullet=i_\bullet \cdot \pi_j$. Therefore it follows that:  
\[\mathsf{D}^\omega[1] = \omega(1)_1 = i_1 = \pi_1 \quad \quad \quad \mathsf{D}^\omega[\pi_j] = \omega(\pi_j)_1 = (i_\bullet \cdot \pi_j)_1 = i_1 \pi_j = \pi_1 \pi_j\]
\textbf{[CD.5]}: By the functoriality of $\omega$, we have that $\omega(fg)_\bullet = \omega(f)_\bullet \ast \omega(g)_\bullet$. By the $\mathcal{D}$-coalgebra structure, we also have that $f = \varepsilon(\omega(f)_\bullet) = \omega(f)_0$. Therefore it follows that: 
\begin{align*}
\mathsf{D}^\omega[fg]&=~ \omega(fg)_1 \\
&=~ \left( \omega(f)_\bullet \ast \omega(g)_\bullet \right)_1 \\
&=~ \mathsf{T}(\omega(f)_\bullet)_0 \omega(g)_1 \\
&=~ \langle \pi_0 \omega(f)_0, \omega(f)_1 \rangle \omega(g)_1 \\
&=~ \langle \pi_0 f, \mathsf{D}^\omega[f] \rangle \mathsf{D}^\omega[g]
\end{align*}
For the remaining two axioms, which involve the higher order derivative $(\mathsf{D}^\omega)^2$, notice that by the $\mathcal{D}$-coalgebra structure, we have the following equality: 
\begin{align*}
(\mathsf{D}^\omega)^{n+m}[f] &= \omega\left((\mathsf{D}^\omega)^n[f] \right)_m\\
&=~ \omega(\omega(f)_n)_m =Ê\left(\mathcal{D}[\omega](\omega(f)_\bullet)_n\right)_m \\
&=~ \left(\delta(\omega(f)_\bullet)_n\right)_m \\
&=~ \left(\mathsf{D}^\omega\right)^n[\omega(f)_\bullet]_m \\
&=~ \omega(f)_{n+m}
\end{align*}
In particular when $n=m=1$, we have that $(\mathsf{D}^\omega)^2[f] = \omega(f)_2$. \\Ê\\
\textbf{[CD.6]}: Here we use {\bf [DS.$3^\prime$]} for $\omega(f)_\bullet$ at $n=1$: 
\[\ell (\mathsf{D}^\omega)^2[f] = \ell \omega(f)_2 = \omega(f)_1 = \mathsf{D}^\omega[f]\]
\textbf{[CD.7]}: Here we use {\bf [DS.$4^\prime$]} for $\omega(f)_\bullet$ at $n=1$: 
\[c (\mathsf{D}^\omega)^2[f] = c \omega(f)_2 = \omega(f)_2 = (\mathsf{D}^\omega)^2[f]\]
\end{proof} 

\begin{lemma}\label{lemfunc2} If $\mathsf{F}: (\mathbb{X}, \omega) \to (\mathbb{Y}, \omega^\prime)$ is a $\mathcal{D}$-coalgebra morphism, then $\mathsf{F}$ is a strict Cartesian differential functor with respect to differential combinators of Proposition \ref{Dcoalgprop2}. 
\end{lemma}
\begin{proof} Straightforward by definition of the differential combinators and $\mathcal{D}$-coalgebra morphisms: 
\[\mathsf{F}(\mathsf{D}^\omega[f]) = \mathsf{F}(\omega(f)_1) = \mathcal{D}[\mathsf{F}](\omega(f)_\bullet)_1 = \omega^\prime(\mathsf{F}(f))_1 = \mathsf{D}^{\omega^\prime}[\mathsf{F}(f)]\]
\end{proof} 

Finally we show that we indeed have an equivalence between $\mathcal{D}$-coalgebras and Cartesian differential categories. 

\begin{theorem}\label{mainthm} The category of $\mathcal{D}$-coalgebras of the comonad $(\mathcal{D}, \delta, \varepsilon)$ is equivalent to the category of Cartesian differential categories and strict Cartesian differential functors between them.  \end{theorem} 
\begin{proof}It is sufficient to show that constructions of Proposition \ref{Dcoalgprop1} and Proposition \ref{Dcoalgprop2} are inverse to each other. Starting with a Cartesian differential category $\mathbb{X}$ with differential combinator $\mathsf{D}$, we have that: 
\[\mathsf{D}^{\omega^\mathsf{D}}[f] = \omega^\mathsf{D}(f)_1 = \mathsf{D}[f]\]
Conversly, let $(\mathbb{X}, \omega)$ be a $\mathcal{D}$-coalgebra, and recall that $(\mathsf{D}^\omega)^{n+m}[f]= \omega(f)_{n+m}$ (as shown in the proof of Proposition \ref{Dcoalgprop2}). Then when $m=0$, we obtain that: 
\[\omega^{\mathsf{D}^\omega}(f)_n = (\mathsf{D}^\omega)^n [f] = \omega(f)_n \]
As $\mathcal{D}$-coalgebra morphisms are precisely strict Cartesian differential functors (Lemma \ref{lemfunc1} and Lemma \ref{lemfunc2}), we obtain the desired equivalence of categories. 
\end{proof} 

As an immediate consequence of Theorem \ref{mainthm}, since the categories of $\mathsf{D}$-sequences are in fact the cofree $\mathcal{D}$-coalgebras, we obtain that: 

\begin{corollary}\label{DseqCDC} For a Cartesian left additive category $\mathbb{X}$, its category of $\mathsf{D}$-sequences $\mathcal{D}[\mathbb{X}]$ is a Cartesian differential category whose differential combinator induced by $\delta$ (as defined in Proposition \ref{Dcoalgprop2}) is precisely given by the differential of pre-$\mathsf{D}$-sequences (as defined in Definition \ref{tandiffseq} (ii)), i.e.:   
\[\mathsf{D}^\delta[f_\bullet] = \mathsf{D}[f_\bullet]\]
 Furthermore, the tangent functor of $\mathcal{D}[\mathbb{X}]$ induced by Proposition \ref{tangentfunctorprop} is precisely given by the tangent of pre-$\mathsf{D}$-sequences as defined in Definition \ref{tandiffseq} (i). 
\end{corollary}

The curious reader may wonder what are the linear maps of $\mathcal{D}[\mathbb{X}]$. Recall that a $\mathsf{D}$-sequence $f_\bullet$ is said to be linear if $\mathsf{D}[f_\bullet] = (i_\bullet \cdot \pi_1) \ast f_\bullet = \pi_1 \cdot f_\bullet$. 

\begin{lemma} A $\mathsf{D}$-sequence $f_\bullet$ is linear if and only if $f_\bullet = i_\bullet \cdot f_0$. 
\end{lemma} 
\begin{proof} $\Rightarrow$: Suppose that $f_\bullet$ is linear. In particular this implies that:
\[f_{n+1} = \mathsf{D}[f_\bullet]_n = (\pi_1 \cdot f_\bullet)_n = \mathsf{P}^n(\pi_1) f_n \]
We now show by induction on $n$ that $f_n = (i_\bullet \cdot f_0)_n$. When $n=0$, we have that $f_0 = i_0 f_0 = (i_\bullet \cdot f_0)_0$. Now suppose the desired equality holds for $k \leq n$, we now show it for $n+1$: 
\begin{align*}
f_{n+1}&=~ \mathsf{P}^n(\pi_1) f_n \\
&=~ \mathsf{P}^n(\pi_1) (i_\bullet \cdot f_0)_n \\
&=~ \mathsf{P}^n(\pi_1) i_n f_0 \\
&=~  \mathsf{P}^n(\pi_1) \underbrace{\pi_1 \hdots \pi_1}_{n \text{ times}} f_0 \\
&=~ \underbrace{\pi_1 \hdots \pi_1}_{n+1 \text{ times}} f_0 \\
&=~ i_{n+1} f_0 \\
&=~ (i_\bullet \cdot f_0)_{n+1}
\end{align*}
$\Leftarrow$: Suppose that $f_\bullet = i_\bullet \cdot f_0$. 
\[\mathsf{D}[f_\bullet] = \mathsf{D}[i_\bullet \cdot f_0] = \mathsf{D}[i_\bullet] \cdot f_0 = (i_\bullet \cdot \pi_1) \cdot f_0 = (\pi_1 \cdot i_\bullet) \cdot f_0 = \pi_1 \cdot (i_\bullet \cdot f_0) = \pi_1 \cdot f_\bullet\]
\end{proof} 

For the comonad $\mathsf{Fa{\grave{a}}}$ of the Fa\`a di Bruno construction, $\mathsf{Fa{\grave{a}}}$-coalgebras are precisely Cartesian differential categories \cite[Theorem 3.2.6]{cockett2011faa}. Then as another consequence of Theorem \ref{mainthm}, $\mathcal{D}$-coalgebras are equivalent to $\mathsf{Fa{\grave{a}}}$-coalgebras, and in particular: 

\begin{corollary}\label{cor1} For a Cartesian left additive category $\mathbb{X}$, $\mathsf{Fa{\grave{a}}}(\mathbb{X})$ (as defined in \cite[Section 2]{cockett2011faa}) and its category of $\mathcal{D}$-sequences $\mathcal{D}[\mathbb{X}]$, are equivalent as Cartesian differential categories. 
\end{corollary}

\section{Conclusion}

The main goal of this paper was to develop an alternative construction to the Fa\`a di Bruno construction for building cofree Cartesian differential categories. In particular, this constructions avoids the combinatorics of the Fa\`a di Bruno formula by instead considering an expression of the higher-order chain rule which involves the tangent functor (\ref{highchainrule}). And as Robert Seely once told the author: ``this construction clears away all the (symmetric) trees that hid the real structure''. 

It is interesting to note that pre-$\mathsf{D}$-sequences and much of their structure (such as composition and differentiation) can be defined for arbitrary categories with finite products -- which provides the possibility of studying differentiation and the chain rule in contexts without an additive structure. Hopefully this new construction will pave the way for future study on cofree Cartesian differential categories. For example, one could study the category of pre-$\mathsf{D}$-sequences and $\mathsf{D}$-sequences for specific categories. Such as what is the category of pre-$\mathsf{D}$-sequences for the category of sets? Or what is the category of $\mathsf{D}$-sequences of the category of commutative monoids?  

Finally, as suggested by one of the reviewers, there is another possible approach involving axiomatizing sequences of the form $(f, \mathsf{T}(f), \mathsf{T}^2(f), \hdots)$, where $\mathsf{T}$ is the tangent functor of a Cartesian differential category. Composition of these sequences, which would correspond to functoriality of the tangent functor ($\mathsf{T}(fg)= \mathsf{T}(f)\mathsf{T}(g)$), is simpler as it is given by point-wise composition. However, the differential of these sequences is no longer given by a simple shift, and the product structure is no longer given point-wise but now involving the permutation $\mathsf{P}(A \times B) \cong \mathsf{P}(A) \times \mathsf{P}(B)$. For the purpose of this paper, $\mathsf{D}$-sequences are the more natural choice as they focus on the Cartesian differential category structure (particularly the differential combinator) rather than the tangent category structure (which in a certain sense ``hides'' the differential combinator).    \\ 

\subparagraph{Acknowledgements:} The author would like thank Robin Cockett and Robert Seely for their support of this project, their editorial suggestions, and useful discussions. The author would also like to thank the reviewers for this paper for their editorial comments and suggestions. 

\bibliographystyle{spmpsci}      

\appendix

\section{Generalized Pre-$\mathsf{D}$-Sequences and Generalized $\mathsf{D}$-Sequences}\label{GCDCsec}

In this appendix, we very briefly explain how to generalize pre-$\mathsf{D}$-sequences and $\mathsf{D}$-sequences to construct cofree generalized Cartesian differential categories. We elected to focus on constructing cofree Cartesian differential categories in this paper instead since the construction is simpler and more enlightening. We begin with the definition of a generalized Cartesian differential category. 

\begin{definition} A \textbf{generalized Cartesian differential category} \cite[Definition 2.1]{cruttwell2017Cartesian} is a category $\mathbb{X}$ with finite products such that:
\begin{enumerate}[{\em (i)}]
\item For each object $A$, there is a chosen commutative monoid $(\mathsf{L}(A), +_A, 0_A)$, where 
\[+_A: \mathsf{L}(A) \times \mathsf{L}(A) \to \mathsf{L}(A) \quad \quad 0_A: \mathsf{1} \to \mathsf{L}(A)\] 
and such that these choices satisfy $\mathsf{L}(\mathsf{L}(A)) = \mathsf{L}(A)$ and $\mathsf{L}(A \times B) = \mathsf{L}(A) \times \mathsf{L}(B)$. 
\item And $\mathbb{X}$ comes equipped with a combinator $\mathsf{D}$ on maps, which written as an inference rule gives: 
\[\infer{\mathsf{D}[f]: A \times \mathsf{L}(A) \to \mathsf{L}(B)}{f: A \to B}\]
such that $\mathsf{D}$ satisfies the equalities found in  \cite[Definition 2.1]{cruttwell2017Cartesian}. 
\end{enumerate}
\end{definition}

In particular, generalized Cartesian differential categories replace the requirement of having a left additive structure with requiring each object comes paired with a commutative monoid. There is also a Fa\`a di Bruno construction for generalized Cartesian differential categories \cite[Section 2.1]{cruttwell2017Cartesian}. Following that Fa\`a di Bruno construction, we consider a generalization of pre-$\mathsf{D}$-sequences and $\mathsf{D}$-sequences where the maps of the sequences are of type $A \times M \times \hdots M \to N$, where $M$ will play the role of the monoid $\mathsf{L}(A)$.  

Let $\mathbb{X}$ be a category with finite products and consider the category $\mathbb{X} \times \mathbb{X}$, whose objects and maps are pairs of objects and maps of $\mathbb{X}$. Define the functor $\tilde{\mathsf{P}}: \mathbb{X} \times \mathbb{X} \to \mathbb{X} \times \mathbb{X}$ on objects as $\tilde{\mathsf{P}}(A, M):= (A \times M, M \times M)$ and on maps as $\tilde{\mathsf{P}}(f,g):= (f \times g, g \times g)$. Now define the functor $\mathsf{U}: \mathbb{X} \times \mathbb{X} \to \mathbb{X}$ on objects as $\mathsf{U}(A, M) := A \times M$ and on maps as $\mathsf{U}(f,g) := f \times g$. 

\begin{definition} For a category $\mathbb{X}$ with finite products, a \textbf{generalized pre-$\mathsf{D}$-sequence} between pairs of objects $(A,M)$ and $(B,N)$, denoted $f_\bullet: (A, M) \to (B, N)$, is a sequence of maps of $f_\bullet := (f_0, f_1, \hdots)$, where $f_0: A \to B$ and $f_{n+1}: \mathsf{U} \left( \tilde{\mathsf{P}}^{n}(A, M) \right) \to N$ for all $n \in \mathbb{N}$. 
\end{definition}

Explicitly, a generalized pre-$\mathsf{D}$-sequence $f_\bullet: (A, M) \to (B, N)$ is a sequence of maps ${f_0: A \to B}$, $f_1: A \times M \to N$, $f_2: A \times M \times M \times M \to N$, etc. Every pre-$\mathsf{D}$-sequence $f_\bullet: A \to B$ is a generalized pre-$\mathsf{D}$-sequence $f_\bullet: (A, A) \to (B,B)$.  Just by the definition, one can see why generalized pre-$\mathsf{D}$-sequences are slightly more trickier to work with then pre-$\mathsf{D}$-sequence. However, the differential, tangent, and composition of generalized pre-$\mathsf{D}$-sequences are defined essentially the same way as they were for pre-$\mathsf{D}$-sequences. And therefore, for a category $\mathbb{X}$ with finite products, we obtain its category of generalized pre-$\mathsf{D}$-sequences $\overline{\mathcal{GD}}[\mathbb{X}]$. It is interesting to note that $\overline{\mathcal{D}}[\mathbb{X}]$ is to $\overline{\mathcal{GD}}[\mathbb{X}]$, what $\mathsf{Fa{\grave{a}}}(\mathbb{X})$ was to $\mathsf{BFa{\grave{a}}}(\mathbb{X})$ (as defined in \cite[Section 2]{cockett2011faa}). 

From here one can construct a comonad $\overline{\mathcal{GD}}$ on $\mathsf{CART}$ given by the category of generalized pre-$\mathsf{D}$-sequences (similar to the comonad defined in Section \ref{predcomsec}). As before, generalized pre-$\mathsf{D}$-sequences are too arbitrary to give cofree generalized Cartesian differential categories. In particular, generalized pre-$\mathsf{D}$-sequences do not require any added requirements on $M$. Generalized $\mathsf{D}$-sequence will require that $M$ comes equipped with a commutative monoid structure, in order to play the role of $\mathsf{L}(A)$. 

\begin{definition} For a category $\mathbb{X}$ with finite products, a \textbf{generalized $\mathsf{D}$-sequence}, denoted by $f_\bullet: \left(A, (M, +_M, 0_M) \right) \to \left(B, (N, +_N, 0_N) \right)$ -- where $A$ and $B$ are arbitrary objects of $\mathbb{X}$, while $(M, +_M, 0_M)$ and $(N, +_N, 0_N)$ are commutative monoids of $\mathbb{X}$ --  is a generalized pre-$\mathsf{D}$-sequence $f_\bullet: (A, M) \to (B,N)$ such that for each $n \in \mathbb{N}$ and $k \leq n$, the following equalities hold: 
\begin{enumerate}[{\bf [GDS.$1^\prime$]}]
\item $\langle 1_A, 0_M \rangle f_1 =0$ and $\mathsf{U} \left(\tilde{\mathsf{P}}^k\left(\langle 1_A, 0_M \rangle, \langle 1_M, 0_M \rangle \right) \right) f_{n+2} = 0 $; 
\item $\mathsf{U} \left(\tilde{\mathsf{P}}^k\left(1_A, +_M \right) \right) f_{n+1} = \left \langle\mathsf{U} \left(\tilde{\mathsf{P}}^k\left(1_A, \pi_0 \right) \right) f_{n+1}, \mathsf{U} \left(\tilde{\mathsf{P}}^k\left(1_A, \pi_1 \right) \right)  f_{n+1} \right \rangle +_M$;
\item $\mathsf{U} \left(\tilde{\mathsf{P}}^k\left(\langle 1_A, 0_M \rangle, \langle 0_M, 1_A \rangle \right) \right) f_{n+2} =   f_{n+1}$;
\item $c f_2 =f_2$ and $\mathsf{U} \left(\tilde{\mathsf{P}}^k\left(c, c \right) \right) f_{n+3} =   f_{n+3}$. 
\end{enumerate}
\end{definition} 

Here, one can clearly see that our simple notation for pre-$\mathsf{D}$-sequences falls apart for generalized $\mathsf{D}$-sequences. Note that generalized $\mathsf{D}$-sequences can be defined for any category with finite products, where $\mathsf{D}$-sequences could only be defined for Cartesian left additive categories. Indeed given a category $\mathbb{X}$, we obtain its category of generalized $\mathsf{D}$-sequences $\mathcal{GD}[\mathbb{X}]$. From here, following the same constructions and arguments as in Section \ref{Dcomsec}, one can see that generalized $\mathsf{D}$-sequences induce a comonad $\mathcal{GD}$ on $\mathsf{CART}$ whose coalgebras are precisely generalized Cartesian differential categories. In particular, in the category of generalized $\mathsf{D}$-sequences, the chosen monoid (of the generalized Cartesian differential category structure) for the object $\left(A, (M, +, 0) \right)$ is $\left(M, (M, +, 0) \right)$.

\end{document}